\pdfoutput=1
\documentclass[10pt]{article}
\usepackage[english]{babel}
\usepackage{amsmath,amsthm}
\usepackage{amsfonts}
\usepackage{hyperref}
\usepackage{bookmark}
\usepackage{geometry}
\usepackage[numbers,sort]{natbib}
\usepackage{indentfirst}
\newtheorem{thm}{Theorem}[section]

\newtheorem{lem}[thm]{Lemma}

\theoremstyle{definition}
\newtheorem{defn}[thm]{Definition}
\theoremstyle{remark}

\numberwithin{equation}{section}

\everymath{\displaystyle}
\begin{document}
\title{Multiplicity and asymptotic behavior of normalized solutions to $p$-Kirchhoff equations}
\author{Jianwen Zhou\thanks{$\textrm{ E-mail:jwzhou@ynu.edu.cn (Jianwen Zhou)}$}, Puming Yang\thanks{$\textrm{ E-mail:yangpuming@stu.ynu.edu.cn (Puming Yang)}$}\\
\small School of Mathematics and Statistics,\\
\small Yunnan  University, Kunming, Yunnan, P.R. China.}%
\date{}%

\maketitle
\begin{abstract}
In this paper, we study a type of $p$-Kirchhoff equation 
$$
	-\left( a+b\int_{\mathbb{R} ^3}{\left| \nabla u \right|^pdx} \right) \varDelta _pu=\lambda \left| u \right|^{p-2}u+\left| u \right|^{q-2}u, x \in \mathbb{R}^3
$$
with the prescribed mass
$$
\left(\int_{\mathbb{R} ^3}{\left| u \right|^{p}dx}\right)^\frac{1}{p} = c > 0
$$
where $a>0, b > 0,\frac{3}{2} <p <3, p < q < p^{\ast}:=\frac{3p}{3-p} $,$\varDelta _pu=div\left( \left| \nabla u \right|^{p-2}\nabla u \right)$ is the $p$-Laplacian of $u$, $\lambda \in \mathbb{R}$ is Lagrange multiplier. We consider both $L^p$-subcritical , $L^p$-critical and $L^p$-supercritical cases. Precisely, in the $L^p$-subcritical and $L^p$-critical cases, we obtain the existence and nonexistence of the normalized solutions for the $p$-Kirchhoff equation. In the $L^p$-supercritical case, we obtain the existence of radial ground sates and multiplicity of radial normalized solutions for the $p$-Kirchhoff equation. Furthermore, we study the asymptotic behavior of normalized solutions when $b \rightarrow 0^+$.
Besides, when $\frac{3}{2} < p \leq 2$, benefit from the uniqueness(up to translations) of optimizer for Gargliardo-Nirenberg inequality, we show the existence and uniqueness of normalized solutions and provide the accurate descriptions. 
\end{abstract}
\section{Introduction and main results}
In this paper, we study a type of $p$-Kirchhoff equation 
\begin{equation}\label{eq1.1}
	-\left( a+b\int_{\mathbb{R} ^3}{\left| \nabla u \right|^pdx} \right) \varDelta _pu=\lambda \left| u \right|^{p-2}u+\left| u \right|^{q-2}u, x \in \mathbb{R}^3
\end{equation}
with the prescribed mass
\begin{equation}\label{eq1.2}
\left(\int_{\mathbb{R} ^3}{\left| u \right|^{p}dx}\right)^\frac{1}{p} = c>0.
\end{equation}
where $a>0, b > 0,\frac{3}{2} <p <3, p < q < p^{\ast}:=\frac{3p}{3-p} $,$\varDelta _pu=div\left( \left| \nabla u \right|^{p-2}\nabla u \right)$ is the $p$-Laplacian of $u$, $\lambda \in \mathbb{R}$ is Lagrange multiplier.
\par
In rencent years, many researchers have focused on studying the p-Kirchhoff type problem, e.g.,\cite{correa2009p,furtado2019multiple,pucci2016existence} and their references therein. \\
\par
  When $p = 2$, it reduces to the following general Kirchhoff equation
\begin{equation}\label{eq1.3}
-\left( a+b\int_{\mathbb{R} ^3}{\left| \nabla u \right|^2dx} \right) \varDelta u=\lambda u+\left| u \right|^{q-2}u, x \in \mathbb{R}^3
\end{equation}
  which is first proposed by G. Kirchhoff in \cite{kirchhoff1883vorlesungen}. It shows the free vibration of elastic strings and describes the motions of moderately large amplitude. Kirchhoff type problem has been widely studied since a functional analysis approach is proposed in the significant work of J.L. Lions \cite{lions1978some}. \\
\par
  When $a=1, b=0$, equation(\ref{eq1.1}) reduces to the following $p$-Laplacian equation
\begin{equation}\label{eq1.4}
	-\varDelta _pu=\lambda \left| u \right|^{p-2}u+\left| u \right|^{q-2}u, x \in \mathbb{R}^3
\end{equation}
The $p$-Laplacian operator $\varDelta _p$ has a strong physical significance. It can be used as models for various physical problems. In fluid mechanics, the value $1<p<2$ (resp.,$p=2$ and $p>2$) corresponds to a pseudoplastic (resp.,Newtonian and diltant) fluid.\\
\par
Recently, many researchers have focused on normalized solutions. Ye \cite{ye2015sharp} studied the existence of constrained minimizers for the Kirchhoff equation (\ref{eq1.3}) and obtained the sharp existence of normalized solutions under different assumptions; Following \cite{ye2015sharp}, Zeng and Zhang \cite{zeng2017existence} obtained the existence and uniqueness of normalized solutions for equation (\ref{eq1.3}). For the $p$-Laplacian equations (\ref{eq1.4}), Lou and Zhang \cite{lou2024multiplicity} studied the multiplicity of normalized solutions under the $L^p$ mass constrant (\ref{eq1.2}). For the $p$-Kirchhoff equation (\ref{eq1.1}), Ren and Lan \cite{ren2024normalized} studied the existence of normalized solutions in the $L^2$ supercritical case.\\
\par
As far as we know, there are few works concerning with the $p$-Kirchhoff equation (\ref{eq1.1}) with the $L^p$ mass constraint (\ref{eq1.2}). Thus we will explore Eqs. (\ref{eq1.1})-(\ref{eq1.2}) in this paper.\\
\par
It is well known that the appearance of nonlocal term $\left(\int_{\mathbb{R} ^3}{\left| \nabla u \right|^pdx}\right)\varDelta_p u$ causes several mathematical difficulties and makes the problem different from the general $p$-Lapalcian case. 
By taking $\lambda$ as a Lagrange multiplier, we consider critical points of the following $C^1$-functional
\begin{equation}\label{eq1.5}
I(u) := \frac{a}{p}\int_{\mathbb{R} ^3}{\left| \nabla u \right|^pdx} + \frac{b}{2p}\left(\int_{\mathbb{R} ^3}{\left| \nabla u \right|^pdx} \right)^2 - \frac{1}{q}\int_{\mathbb{R} ^3}{\left| u \right|^pdx}.
\end{equation}
constrained on
$$
S(c) := \left\{u \in W^{1,p}(\mathbb{R}^3) : \left(\int_{\mathbb{R} ^3}{\left| u \right|^{p}dx}\right)^\frac{1}{p} = c>0 \right\},
$$
where
$$
W^{1,p}(\mathbb{R}^3) := \left\{ u \in L^p(\mathbb{R}^3) : \int_{\mathbb{R} ^3}{\left| \nabla u \right|^{p}dx} < \infty \right\}.
$$
We know the critical points of $I|_{S(c)}$ are weak solutions to (\ref{eq1.1})-(\ref{eq1.2}) and define
\begin{equation}\label{eq1.6}
i(c) := \underset{u \in S(c)}{\text{inf}}I(u).
\end{equation}
In the $L^p$ subcritical and critical cases, we give the existence of normalized solutions to the $p$-Kirchhoff Eqs. (\ref{eq1.1})-(\ref{eq1.2}):
\begin{thm}\label{thm1.1}
(1) Let $p<q<p+\frac{2p^2}{3}$, then
\par (i) when $p<q<p+\frac{p^2}{3}$, i(c) has a minimizer for any $c>0$.
\par (ii) when $q=p+\frac{p^2}{3}$, i(c) has a minimizer if and only if $c> a^{\frac{3}{p^2}}\left| Q \right|_p$, where $|\cdot|_s$ denote the norm of $L^s(\mathbb{R}^3)$ and $Q \in W^{1,p}(\mathbb{R}^3)$ is a ground-state solution to the following equation
$$
-\frac{3(q-p)}{p^2}\varDelta _pu + \left(1+\frac{(p-3)(q-p)}{p^2}\right)\left| u \right|^{p-2}u = \left| u \right|^{q-2}u, x \in \mathbb{R}^3.
$$
\par (iii) when $p+\frac{p^2}{3}<q<p+\frac{2p^2}{3}$, there exist $c^{\ast} > 0$  such that i(c) has a minimizer if and only if $c \geq c^{\ast}$, where
$$
c^{\ast} = \left[ p\left| Q \right|_{p}^{q-p}\left(\frac{ap}{2p^2-3q+3p)}\right)^{\frac{2p^2-3q+3p}{p^2}}\left(\frac{bp}{6pq-8p^2}\right)^{\frac{3q-3p-p^2}{p^2}} \right] ^{\frac{p}{pq-3q+3p}}.
$$
\par Moreover, for any $p<q<p+\frac{2p^2}{3}$ and $c > 0$, if $i(c)$ has a minimzer, then Eqs.(\ref{eq1.1})-(\ref{eq1.2}) has a couple of solutions $(u,\lambda) \in W^{1,p}(\mathbb{R}^3) \times \mathbb{R}$ such that $u \in S(c)$ and $\lambda < 0$.
\par (2)When $q=p+\frac{2p^2}{3}$, i(c) has no minimizer for all $c>0$.
\end{thm}

In the following three conclusions, we study the Eqs.(\ref{eq1.1})-(\ref{eq1.2}) in the $L^p$ mass supercritical case. In this case, the energy functional $I$ constrained on $S(c)$ is not bounded from below which leads to many difficulties. In order to overcome the lack of compactness, we work on the radially symmetric functionals space $W_{r}^{1,p}(\mathbb{R}^3)$ and consider the following minimization problem:
\begin{equation}\label{eq1.7}
m(c) := \underset{u \in M(c)}{\text{inf}} I(u),
\end{equation}
where
$$
M(c) := \{u \in S_{r}(c):P(u)=0 \},
$$
$$
S_r(c) := \{|u|_{p} = c : u\in W_{r}^{1,p}(\mathbb{R}^3)\},
$$
and
$$
P(u) := a\left| \nabla u \right|_{p}^{p} + b\left| \nabla u \right|_{p}^{2p} - \frac{3(q-p)}{pq}\left| u \right|_{q}^{q} = 0
$$
is associated Pohozaev identity of Eqs(\ref{eq1.1})-(\ref{eq1.2}).

\begin{defn}\label{defn1.2}
 We call $(u,\lambda)$ is a ground-state solution to (\ref{eq1.1})-(\ref{eq1.2}), if $(u,\lambda)$ is a solution to (\ref{eq1.1})-(\ref{eq1.2}) and $u$ minimizes the energy functional $I$ among all solutions of $L^{p}$-norm.
\end{defn}

\begin{thm}\label{thm1.3} Let $p+\frac{2p^2}{3} < q < p^{\ast}$, for any $c > 0$, (\ref{eq1.1})-(\ref{eq1.2}) has a couple of solution $(u, \lambda) \in  W_{r}^{1,p}(\mathbb{R}^3) \times \mathbb{R}$ with $I(u) = m(c)$ and $\lambda < 0$. Moreover, $u$ is a radial ground state of (\ref{eq1.1})-(\ref{eq1.2}).
\end{thm}

Next, we show the multiplicity of normalized solutions to (\ref{eq1.1})-(\ref{eq1.2}):
\begin{thm}\label{thm1.4} Let $p+\frac{2p^2}{3} < q < p^{\ast}$ and $c > 0$. Then there exists a sequence $\{(u_n,\lambda_n)\}$ of solutions of (\ref{eq1.1})-(\ref{eq1.2}) such that $u_n$ is radial and $\lambda_n < 0$ for any $n \in \mathbb{N}^{+}$. Moreover, we have $\Vert u_n \Vert_{ W_{r}^{1,p}(\mathbb{R}^3) } \rightarrow +\infty$ and $I(u_n) \rightarrow +\infty$ as $n \rightarrow +\infty$.
\end{thm}

Based on Theorem \ref{thm1.4}, we study the asymptotic behavior of normalized solutions to (\ref{eq1.1})-(\ref{eq1.2}) when $b \rightarrow 0^+$. We obtain the following result:
\begin{thm}\label{thm1.5} Let $p+\frac{2p^2}{3} < q < p^{\ast}$, $c > 0$ and $\{(u_n^b,\lambda_n^b)\} \subset S_r(c) \times \mathbb{R}^{-}$ is obtained in Theorem \ref{thm1.4}. Then, for any $\{b_m\} \rightarrow 0^{+}(n \rightarrow +\infty)$, there is a subsequence $\{b_{m_l}\}$ such that for each $n \in \mathbb{N}^{+}$, $u_n^{b_{m_l}} \rightarrow u_n^0$ in $W_{r}^{1,p}(\mathbb{R}^3)$ and $\lambda_n^{b_{m_l}} \rightarrow \lambda_n^0$ in $\mathbb{R}$ as $l \rightarrow +\infty$, where $\{(u_n^0,\lambda_n^0)\} \subset S_r(c) \times \mathbb{R}^{-}$ is a sequence of couples of weak solutions to the following equation
\begin{equation}\label{eq1.8}
-a\varDelta_p u - \lambda|u|^{p-2}u = |u|^{q-2}u, \enspace in \enspace \mathbb{R}^3.
\end{equation}
\end{thm}

As is well known, if $1 < p \leq 2$, the ground state to (\ref{eq2.2}) is unique (up to translation), see \cite{serrin2000uniqueness}. However, if $p > 2$, the uniqueness have not been completely solved. When $1 < p \leq 2$, we obtain an accurate description of constrained minimizer of $I$ for $i(c)$ when $p< q \leq p+\frac{p^2}{3}$ and a accurate solution to (\ref{eq1.1})-(\ref{eq1.2}) when $p+\frac{2p^2}{3} \leq q < p^{\ast}$, which is stated in the following two results.

\begin{thm}\label{thm1.6} We assume that $\frac{3}{2}< p \leq 2$, then
\par
(1) When $p < q < p+\frac{p^2}{3}$, $i(c)$ has a unique minimizer (up to translations), which is the form of $u_c = \frac{c\mu_q^{\frac{3}{p}}}{|Q|_p}Q(\mu_q x)$, where $\mu_q = \frac{t_q^{\frac{1}{p}}}{c}$ with $t_q$ being the unique minimum point of the function
\begin{equation}\label{eq1.9}
f_q(t) = \frac{a}{p}t+\frac{b}{2p}t^2-\frac{c^{q-\frac{3(q-p)}{p}}}{p|Q|_p^{q-p}}t^{\frac{3(q-p)}{p^2}}, t \in (0,+\infty).
\end{equation}
\par
(2) When $q = p+\frac{p^2}{3}$, if $c> a^{\frac{3}{p^2}}\left| Q \right|_p$, then $i(c)$ has a unique minimizer (up to translations)
\begin{equation}\label{eq1.10}
u_c = \frac{c\mu_q^{\frac{3}{p}}}{|Q|_p}Q(\mu_q x) \enspace where \enspace \mu_q = \frac{t_q^{\frac{1}{p}}}{c} \enspace with \enspace t_q = \frac{c^{\frac{p^2}{3}}-ap|Q|_p^{\frac{p^2}{3}}}{bp|Q|_p^{\frac{p^2}{3}}}.
\end{equation}
Moreover, $i(c) = -\frac{b}{2p}\left(\frac{c^{\frac{p^2}{3}}-ap|Q|_p^{\frac{p^2}{3}}}{bp|Q|_p^{\frac{p^2}{3}}}\right)^2$.
\end{thm}

When $q \geq p+\frac{2p^2}{3}$, Theorem \ref{thm1.1} tells that $i(c)$ has no minimizer. Thus, motivated by \cite{zeng2017existence}, we investigate the mountain pass type critical points for $I(\cdot)$ on $S(c)$.

\begin{defn}\label{defn1.7} Given $c>0$, the functional $I(\cdot)$ is said to have mountain pass geometry on $S(c)$ if there exists $K(c) > 0$ such that
\begin{equation}\label{eq1.11}
\gamma(c) := \underset{h \in \Gamma(c)}{\text{inf}} \underset{t \in [0,1]}{\text{max}}I(h(t)) > \text{max}\{I(h(0)),I(h(1))\}
\end{equation}
holds in the set $\Gamma(c)=\{h \in C([0,1];S(c)): h(0) \in A_{K(c)}, I(h(1)) < 0\}$, where $A_{K(c)} = \{u \in S(c) : |\nabla u|_p^p \leq K(c)\}$.
\end{defn}

Next, when $\frac{3}{2} <p \leq 2$, under the following assumption
\begin{equation}\label{eq1.12}
q > p+\frac{2p^2}{3}, or \enspace q = p+\frac{2p^2}{3} \enspace and \enspace c > \left(\frac{b\left| Q\right|_{p}^{\frac{2p^2}{3}}}{2}\right)^{\frac{3}{2p^2-3p}},
\end{equation}
we investigate some properties of (\ref{eq1.11}) by some energy estimates. Moreover, we show that if $u_c \in S(c)$ is a critical point of $I(\cdot)$ on the level $\gamma(c)$, then it is a scaling of $Q(x)$. Still let $f_q(\cdot)$ be given by (\ref{eq1.9}) and note that it has a unique maximum point in $(0,+\infty)$ when (\ref{eq1.12}) holds. Then, we have the following theorem.

\begin{thm}\label{thm1.8} When $\frac{3}{2} <p \leq 2$, assume (\ref{eq1.12}) holds and let $\overline{t}_q$ be the unique maximum point of $f_q(t)$ in $(0,+\infty)$. Then $\gamma(c)= f_q(\overline{t}_q)$ and it can be attained by $\overline{u}_c = \frac{c\overline{\mu}_q^{\frac{3}{p}}}{|Q|_p}Q(\overline{\mu}_q x)$ where $\overline{\mu}_q = \frac{\overline{t}_q^{\frac{1}{p}}}{c}$. Meanwhile, $\overline{u}_c$ is a solution of (\ref{eq1.1})-(\ref{eq1.2}) for some $\lambda \in \mathbb{R}$.
\end{thm}

The paper is organized as follows. In Section 2, we present some preliminary results. In Section 3, we study the existence of minimizers for $i(c)$ in the $L^p$-subcritical and critical cases. In Section 4, we show the existence of normalized ground state in the $L^p$-supercritical case. In Section 5, we study the multiplicity and asymptotic behavior of normalized solutions in the $L^p$-supercritical case. In Section 6, we prove the existence and uniqueness of normalized solutions when $\frac{3}{2} < p \leq 2$.

\section{Preliminaries}
In this section, we introduce various preliminary results.

\begin{lem}\label{lem2.1} (Gagliardo-Nirenberg inequality of $p$-Laplacian type, \cite{agueh2008sharp}) For $p \in (\frac{3}{2},3)$ and any $q \in [p,p^{\ast})$, it holds that
\begin{equation}\label{eq2.1}
\left| u \right|_q \leq \left(\frac{q}{p}\frac{1}{\left| Q \right|_{p}^{q-p}}\right)^{\frac{1}{q}} \left| \nabla u \right|_{p}^{\frac{3(q-p)}{qp}} \left| u \right| _{p}^{1-\frac{3(q-p)}{qp}},
\end{equation}
where $\left| u \right|_s$ is the norm of $L^t(\mathbb{R}^3)$ and up to translations, $Q \in W^{1,p}(\mathbb{R}^3)$ is a ground-state solution to the following equation(\cite{weinstein1982nonlinear})
\begin{equation}\label{eq2.2}
-\frac{3(q-p)}{p^2}\varDelta _pu + \left(1+\frac{(p-3)(q-p)}{p^2}\right)\left| u \right|^{p-2}u = \left| u \right|^{q-2}u, x \in \mathbb{R}^3.
\end{equation}
Furthermore, when $1< p \leq2$, the equality holds if and only if $u(x) = \alpha Q(\beta x)$ for some $\alpha,\beta \in \mathbb{R}\backslash \{0\}$.
\end{lem}

\begin{lem}\label{lem2.2} Let $p \in (\frac{3}{2},3)$, $q \in (p,p^{\ast})$, if $(u,\lambda) \in W^{1,p}(\mathbb{R}^3) \times \mathbb{R}$ is a weak solution of (\ref{eq1.1})-(\ref{eq1.2}), then the following identity
\begin{equation}\label{eq2.3}
P(u) = a\left| \nabla u \right|_{p}^{p} + b\left| \nabla u \right|_{p}^{2p} - \frac{3(q-p)}{pq}\left| u \right|_{q}^{q} = 0
\end{equation}
holds. Moreover, it can be obtained that $\lambda < 0$.
\end{lem}
\begin{proof}
the proof is similar to the Lemma 2.3 of \cite{baldelli2022normalized} and Lemma 2.6 of \cite{ye2015sharp}, so we omit it.
\end{proof}

Next, we introduce a useful elementary inequality.
\begin{lem}\label{lem2.3} There exists $c_s > 0$ such that for any $x,y \in \mathbb{R}^3$,
\begin{equation}\label{eq2.4}
\begin{cases}
\left< \left| x \right|^{s-2}x - \left| y \right|^{s-2}y, x-y \right>_{\mathbb{R}^3} \geq c_s\left| x-y \right|^s & for \quad 2 \leq s < 3,\\
(\left| x \right| + \left| y \right|)^{2-s}\left< \left| x \right|^{s-2}x - \left| y \right|^{s-2}y, x-y \right>_{\mathbb{R}^3} \geq c_s\left| x-y \right|^2 & for\quad 1< s <2,
\end{cases}
\end{equation}
where $\left< \cdot,\cdot \right>_{\mathbb{R}^3}$ denotes the standard inner product in $\mathbb{R}^3$.
\end{lem}

For $u,v \in W^{1,p}(\mathbb{R}^3)$, it follows from (\ref{eq2.4}) that ,for $2 \leq p < 3$,
\begin{equation}\label{eq2.5}
c_p\int_{\mathbb{R}^3}{|\nabla u - \nabla v|^{p}dx} \leq \int_{\mathbb{R}^3}{(|\nabla u|^{p-2}\nabla u-|\nabla v|^{p-2}\nabla v)\nabla (u-v) dx},
\end{equation}
and for $ \frac{3}{2} < p < 2$,
\begin{equation}\label{eq2.6}
\begin{split}
& c_{p}^{\frac{p}{2}}\int_{\mathbb{R}^3}{|\nabla u - \nabla v|^{p}dx} \\
& \leq \int_{\mathbb{R}^3}{(T(u,v))^{\frac{p}{2}}(|\nabla u| + |\nabla v|)^{\frac{p(2-p)}{2}}dx}\\
& \leq \left(\int_{\mathbb{R}^3}{(|\nabla u|^{p-2}\nabla u-|\nabla v|^{p-2}\nabla v)\nabla (u-v) dx}\right)^{\frac{p}{2}}\left(\int_{\mathbb{R}^3}{(|\nabla u| + |\nabla v|)^{p}dx}\right)^{\frac{2-p}{2}},
\end{split}
\end{equation}
where
$$
T(u,v) = \left< \left| \nabla u \right|^{p-2}\nabla u - \left| \nabla v \right|^{p-2}\nabla v, \nabla u- \nabla v \right>_{\mathbb{R}^3},
$$
note that the domain of integration in (\ref{eq2.5})-(\ref{eq2.6}) can be replaced by any open domain of $\mathbb{R}^3$.

\begin{lem}\label{lem2.4} (\cite{lions1983nonlinear},Lemma 3) Let $I \in C^1(W^{1,p}(\mathbb{R}^3),\mathbb{R})$. If $\{ u_n \} \subset S(c)$ is bounded in $W^{1,p}(\mathbb{R}^3)$, then
\begin{equation}
\begin{split}
& (I|_{S(c)})^{\prime} (u_n) \rightarrow 0 \quad \text{in} \quad \left(W^{1,p}(\mathbb{R}^3)\right)^{\prime}  \quad as \quad n \rightarrow \infty \\
& \Leftrightarrow 
I^{\prime}(u_n) - \frac{1}{c^p}\left< I^{\prime}(u_n) , u_n \right>\left| u_n \right|^{p-2}u_n \rightarrow 0 \quad \text{in} \quad \left(W^{1,p}(\mathbb{R}^3)\right)^{\prime} \quad \text{as} \quad n \rightarrow \infty. \\
\end{split}
\nonumber
\end{equation}
\end{lem}

\section{Proof of Theorem 1.1}
In this section, we study the case when $p < q \leq p+\frac{2p^2}{3}$. We first give a few of important lemmas before the proof of Theorem \ref{thm1.1}.

\begin{lem}\label{lem3.1}
Assume that $p<q<p+\frac{2p^2}{3}$, then
\par (1) For each $c>0$, $i(c)$ is well defined and $i(c) \leq 0$.
\par (2) If $p<q<p+\frac{p^2}{3}$,  then $i(c)<0$ for all $c>0$.
\par (3) If $q=p+\frac{p^2}{3}$, then $i(c) < 0 $ for each $c>a^{\frac{3}{p^2}}\left| Q \right|_p$ and $i(c) = 0$ for each $0 < c \leq a^{\frac{3}{p^2}}\left| Q \right|_p$.
\par (4) If $p+\frac{p^2}{3}<q<p+\frac{2p^2}{3}$, then $i(c)<0$ for each $c>c^{\ast}$ and $i(c)=0$ for each $c \leq c^{\ast}$,
where
$$
c^{\ast} = \left[ p\left| Q \right|_{p}^{q-p}\left(\frac{ap}{2p^2-3q+3p)}\right)^{\frac{2p^2-3q+3p}{p^2}}\left(\frac{bp}{6pq-8p^2}\right)^{\frac{3q-3p-p^2}{p^2}} \right] ^{\frac{p}{pq-3q+3p}}.
$$
\end{lem}

\begin{proof}
(1) For any $c>0$ and $u \in S(c)$, by the Gagaliardo-Nirenberg inequality in Lemma \ref{lem2.1}, it is obtained that

\begin{equation}\label{eq3.1}
I(u) \geq \frac{a}{p}\left| \nabla u \right|_{p}^{p} + \frac{b}{2p}\left| \nabla u \right|_{p}^{2p} - \frac{c^{q-\frac{3(q-p)}{p}}}{p\left| Q \right|_{p}^{q-p}} \left| \nabla u \right|_{p}^{\frac{3(q-p)}{p}}.
\end{equation}

Since $0 < \frac{3(q-p)}{p} < 2p$, $I$ is bouned from below on $S(c)$. Set $u_t(x) := t^{\frac{3}{p}}u(tx)$, $t>0$, then $u_t \in S(c)$ and
\begin{equation}\label{eq3.2}
i(c) \leq I(u_t) = t^{p}\frac{a}{p}\left| \nabla u \right|_{p}^{p} + t^{2p}\frac{b}{2p}\left| \nabla u \right|_{p}^{2p} - t^{\frac{3(q-p)}{p}}\frac{1}{q}\left| u \right|_{q}^{q} \rightarrow 0 \quad as\quad t \rightarrow 0.
\end{equation}
Hence, $i(c) \leq 0$ for all $c > 0$.
\par
(2) If $p<q<p+\frac{p^2}{3}$, then $0 < \frac{3(q-p)}{p} < p$, so (\ref{eq3.2}) implies that $i(c)<0$ for all $c>0$.
\par
(3) If $q=p+\frac{p^2}{3}$, similar to (\ref{eq3.1}), we see that
\begin{equation}
\begin{split}
I(u) 
&\geq \frac{a}{p}\left| \nabla u \right|_{p}^{p} + \frac{b}{2p}\left| \nabla u \right|_{p}^{2p} - \frac{c^{\frac{p^2}{3}}}{p\left| Q \right|_{p}^{\frac{p^2}{3}}} \left| \nabla u \right|_{p}^{p}.\\
&= \frac{b}{2p}\left| \nabla u \right|_{p}^{2p} + \frac{a}{p}\left| \nabla u \right|_{p}^{p}\left(1-\frac{c^{\frac{p^2}{3}}}{a\left| Q \right|_{p}^{\frac{p^2}{3}}}\right).
\end{split}
\nonumber
\end{equation}
Thus, when $0 < c \leq a^{\frac{3}{p^2}}\left| Q \right|_p$, we have $i(c) \geq 0$. By $i(c) \leq 0$, it is obtained that $i(c) = 0$ for all $0 < c \leq a^{\frac{3}{p^2}}\left| Q \right|_p$.
\par
By (\ref{eq2.2}) and corresponding Pohozaev identity, we have

\begin{equation}\label{eq3.3}
\left| \nabla Q \right|_{p}^{p} = \left| Q \right|_{p}^{p} = \frac{p}{q}\left| Q \right|_{q}^{q}.
\end{equation}

When $c > a^{\frac{3}{p^2}}\left| Q \right|_p$, set 
\begin{equation}\label{eq3.4}
Q_t(x) := \frac{c}{\left| Q \right|_{p}}t^{\frac{3}{p}}Q(tx), \quad t>0,
\end{equation}
hence, $Q_t(x) \in S(c)$. By (\ref{eq3.3}), we have
\begin{equation}\label{eq3.5}
\begin{split}
I(Q_t) 
& = \frac{a}{p}(ct)^p+\frac{b}{2p}(ct)^{2p}-\frac{c^qt^p}{p\left| Q \right|_{p}^{q-p}}\\
& = \frac{b}{2p}(ct)^{2p}+\frac{a(ct)^p}{p}\left(1-\frac{c^{\frac{p^2}{3}}}{a\left| Q \right|_{p}^{\frac{p^2}{3}}}\right).
\end{split}
\end{equation}
Thus, $I(Q_t) < 0$ as $t \rightarrow 0$. Since $i(c) \leq I(Q_t)$, we conclude that $i(c)<0$ when $c > a^{\frac{3}{p^2}}\left| Q \right|_p$.
\par
(4) If $p+\frac{p^2}{3}<q<p+\frac{2p^2}{3}$, similar to (\ref{eq3.5}), we have
\begin{equation}\label{eq3.6}
I(Q_t) = \frac{a}{p}(ct)^p+\frac{b}{2p}(ct)^{2p}-\frac{c^q}{p\left| Q \right|_{p}^{q-p}}t^{\frac{3(q-p)}{p}}.
\end{equation}
By Young's inequality, we have
\begin{equation}\label{eq3.7}
\frac{a}{p}(ct)^p+\frac{b}{2p}(ct)^{2p} 
\geq 
\left(\frac{a}{pp_1}c^{p}t^{p}\right)^{p_1}\left(\frac{b}{2pp_2}c^{2p}t^{2p}\right)^{p_2}
= t^{p+pp_2}c^{p+pp_2}\left(\frac{a}{pp_1}\right)^{p_1}\left(\frac{b}{2pp_2}\right)^{p_2},
\end{equation}
where $0<p_1,p_2 <1$, $p_1+p_2=1$ and the equality holds if and only if $\frac{a}{pp_1}c^{p}t^{p} = \frac{b}{2pp_2}c^{2p}t^{2p}$.
Choosing $p_1=2-\frac{3(q-p)}{p^2}$ and $p_2=\frac{3(q-p)}{p^2}-1$. Since $p+\frac{p^2}{3}<q<p+\frac{2p^2}{3}$, we have $0<p_1,p_2 <1$ and $p_1+p_2=1$. Then, by (\ref{eq3.6}), we have
\begin{equation}
\begin{split}
& I(Q_t) \\
& \geq c^{\frac{3(q-p)}{p}}t^{\frac{3(q-p)}{p}}\left(\frac{ap}{2p^2-3q+3p)}\right)^{\frac{2p^2-3q+3p}{p^2}}\left(\frac{bp}{6pq-8p^2}\right)^{\frac{3q-3p-p^2}{p^2}} - \frac{c^q}{p\left| Q \right|_{p}^{q-p}}t^{\frac{3(q-p)}{p}}\\
& = \frac{c^{\frac{3(q-p)}{p}}t^{\frac{3(q-p)}{p}}}{p\left| Q \right|_{p}^{q-p}}\left[(c^{\ast})^{q-\frac{3(q-p)}{p}} - c^{q-\frac{3(q-p)}{p}}\right].\\
\end{split}
\nonumber
\end{equation}
If $c > c^{\ast}$, choosing $t_{0}^{p} = \frac{2p_2a}{c^pp_1b}$ such that $\frac{a}{pp_1}c^{p}t^{p} = \frac{b}{2pp_2}c^{2p}t^{2p}$. Then,
$$
i(c) \leq I(Q_t) = \frac{c^{\frac{3(q-p)}{p}}t^{\frac{3(q-p)}{p}}}{p\left| Q \right|_{p}^{q-p}}\left[(c^{\ast})^{q-\frac{3(q-p)}{p}} - c^{q-\frac{3(q-p)}{p}}\right] < 0.
$$
If $0<c \leq c^{\ast}$, then for any $u \in S(c)$, by Lemma \ref{lem2.1} and (\ref{eq3.7}), we have
\begin{equation}
\begin{split}
I(u) & \geq \frac{a}{p}\left| \nabla u \right|_{p}^{p} + \frac{b}{2p}\left| \nabla u \right|_{p}^{2p} - \frac{c^{q-\frac{3(q-p)}{p}}}{p\left| Q \right|_{p}^{q-p}} \left| \nabla u \right|_{p}^{\frac{3(q-p)}{p}}\\
& \geq  \frac{1}{p\left| Q \right|_{p}^{q-p}} \left| \nabla u  \right|_{p}^{\frac{3(q-p)}{p}}\left[(c^{\ast})^{q-\frac{3(q-p)}{p}} - c^{q-\frac{3(q-p)}{p}}\right] \geq 0.\\
\end{split}
\nonumber
\end{equation}
Thus $i(c) \geq 0$, combining with $i(c) \leq 0$, we conclude that $i(c)=0$.

\end{proof}

\begin{lem}\label{lem3.2}
If $q=p+\frac{2p^2}{3}$, then $i(c) = -\infty$ when $c > \left(\frac{b\left| Q\right|_{p}^{\frac{2p^2}{3}}}{2}\right)^{\frac{3}{2p^2-3p}}$ and $i(c) = 0$ when $c \leq \left(\frac{b\left| Q\right|_{p}^{\frac{2p^2}{3}}}{2}\right)^{\frac{3}{2p^2-3p}}$
\end{lem}

\begin{proof}
When $c > \left(\frac{b\left| Q\right|_{p}^{\frac{2p^2}{3}}}{2}\right)^{\frac{3}{2p^2-3p}}$, similar to (\ref{eq3.5}), we have
\begin{equation}
\begin{split}
I(Q_t) 
& = \frac{a}{p}(ct)^p+\frac{b}{2p}(ct)^{2p}-\frac{c^qt^p}{p\left| Q \right|_{p}^{q-p}}\\
& = \frac{a}{p}(ct)^p+\frac{b}{2p}(ct)^{2p}\left(1-\frac{2c^{\frac{2p^2-3p}{3}}}{b\left| Q \right|_{p}^{\frac{2p^2}{3}}}\right),\\
\end{split}
\nonumber
\end{equation}
Thus $I(Q_t) \rightarrow -\infty$ as $t \rightarrow \infty$, which implies $i(c) = -\infty$.
\par
When $c \leq \left(\frac{b\left| Q\right|_{p}^{\frac{2p^2}{3}}}{2}\right)^{\frac{3}{2p^2-3p}}$, for any $u \in S(c)$, by Lemma \ref{lem2.1}, we have
\begin{equation}
\begin{split}
I(u) & \geq \frac{a}{p}\left| \nabla u \right|_{p}^{p} + \frac{b}{2p}\left| \nabla u \right|_{p}^{2p} - \frac{c^{q-\frac{3(q-p)}{p}}}{p\left| Q \right|_{p}^{q-p}} \left| \nabla u \right|_{p}^{\frac{3(q-p)}{p}}\\
& = \frac{a}{p}\left| \nabla u \right|_{p}^{p} + \frac{b}{2p}\left| \nabla u \right|_{p}^{2p}\left(1-\frac{2c^{\frac{2p^2-3p}{3}}}{b\left| Q \right|_{p}^{\frac{2p^2}{3}}}\right) \geq 0,\\
\end{split}
\nonumber
\end{equation}
which implies that $i(c) \geq 0$. Since (\ref{eq3.2}) still holds when $q=p+\frac{2p^2}{3}$, we have $i(c) \leq 0$ for any $c>0$. So we conclude that $i(c)=0$.

\end{proof}

\begin{lem}\label{lem3.3} For each $p < q <p+\frac{2p^2}{3}$, the function $c \mapsto i(c)$ is continuous on $(0,+\infty)$.
\end{lem}
\begin{proof}
The proof follows from Lemma \ref{lem3.1} (1) and is similar to the Lemma 3.3 of \cite{baldelli2022normalized}, so we omit it.
\end{proof}

\begin{lem}\label{lem3.4} For each $p < q <p+\frac{2p^2}{3}$, if $c$ satisfies the following condition:
\begin{equation}\label{eq3.8}
\begin{cases}
	c>0, & \text{when} \quad p < q <p+\frac{p^2}{3},\\
    c>a^{\frac{3}{p^2}}\left| Q \right|_p, & \text{when} \quad q=p+\frac{p^2}{3},\\
	c>c^{\ast},&	\text{when} \quad p+\frac{p^2}{3} < q <p+\frac{2p^2}{3},\\
\end{cases}
\end{equation}
where $c^{\ast}$ is defined in Lemma \ref{lem3.1}, then
$$
i(c) < i(\alpha) + i(\sqrt[p]{c^p-\alpha^p}), \enspace \text{for} \enspace \text{any} \enspace 0<\alpha < c.
$$
\end{lem}

\begin{proof}
We first claim that
\begin{equation}\label{eq3.9}
i(\theta c) < \theta^{p}i(c), \enspace \text{for} \enspace \text{any} \enspace \theta > 1.
\end{equation}
By Lemma \ref{lem3.1}, we have $i(c) < 0$ when $c$ satisfies (\ref{eq3.8}). Let $\{u_n\} \subset S(c)$ be a minimizing sequence for $i(c)$. For any $\theta > 1$, set $u_{n}^{\theta} = u(\theta^{-\frac{p}{3}}x)$, then $u_{n}^{\theta} \in S(\theta c)$ and

\begin{equation}
\begin{split}
I(u_{n}^{\theta}) & = \theta^p I(u_n) + \left(\theta^{p-\frac{p^2}{3}}-\theta^p\right)\frac{a}{p}\left| \nabla u_n \right|_{p}^{p} + \left(\theta^{2p-\frac{2p^2}{3}}-\theta^p\right)\frac{b}{2p}\left| \nabla u_n \right|_{2p}^{p} \\
& \leq \theta^p I(u_n).\\
\end{split}
\nonumber
\end{equation}
Let $n \rightarrow \infty$, we have $i(\theta c) \leq \theta^{p}i(c)$, with equality if and only if $\left| \nabla u_n \right|_{p}^{p} \rightarrow 0$ as $n \rightarrow \infty$. By (\ref{eq2.1}), we have $\left| u_n \right|_{q}^{q} \rightarrow 0$ when $\left| \nabla u_n \right|_{p}^{p} \rightarrow 0$ as $n \rightarrow \infty$. Then we obtain $I(u_n) \rightarrow 0$ as $n \rightarrow \infty$ which contradicts $i(c) < 0$. Thus, we have $i(\theta c) < \theta^{p}i(c)$. For any $0<\alpha<c$, apply (\ref{eq3.9}) with $\theta = \frac{c}{\alpha} > 1$ and $\theta = \frac{c}{\sqrt[p]{c^p-\alpha^p}} > 1$ respectively, we get
$$
i(c) = \frac{\alpha^p}{c^p}i\left(\frac{c}{\alpha}\alpha\right) + \frac{c^p-\alpha^p}{c^p}i\left(\frac{c}{\sqrt[p]{c^p-\alpha^p}}\sqrt[p]{c^p-\alpha^p}\right) < i(\alpha)+i(\sqrt[p]{c^p-\alpha^p}).
$$
\end{proof}

\begin{lem}\label{lem3.5}
For each $p < q <p+\frac{2p^2}{3}$ and any $c > 0$, let $\{ w_n \} \subset S(c)$ be a minimizing sequence for $i(c)$, then $i(c)$ possesses another minimizing sequence $\{ v_n \} \subset S(c)$ such that
\begin{equation}
\Vert v_n - w_n \Vert_ {W^{1,p}(\mathbb{R}^3)} \rightarrow 0, (I|_{S(c)})^{\prime}(v_n) \rightarrow 0 \quad \text{as} \enspace n \rightarrow \infty.
\nonumber
\end{equation}
\end{lem}

\begin{proof}
When $p < q <p+\frac{2p^2}{3}$, by Lemma 3.1 (1), we know that $I$ is bounded from below on $S(c)$. By Ekeland's variational principle(\cite{willem2012minimax}, Theorem 2.4), we can get a new minimizing sequence $\{ v_n \} \subset S(c)$ for $i(c)$ such that $\Vert v_n - w_n \Vert_ {W^{1,p}(\mathbb{R}^3)} \rightarrow 0$, which is also a Palais-Smale sequence for $I|_{S(c)}$.
\end{proof}

\begin{proof}[Proof of Theorem \ref{thm1.1}]
We complete the proof in four parts.
\par
At fist, we study the case when $p<q<p+\frac{p^2}{3}$:
\par
(1) For $p<q<p+\frac{p^2}{3}$, by Lemma \ref{lem3.1}, we have $i(c) < 0$ for any $c > 0$. Let $\{ w_n \} \subset S(c)$ be a minimizing sequence for $i(c)$. By Lemma \ref{lem3.5}, there exists a minimizing sequence $\{ v_n \} \subset S(c)$ for $i(c)$ such that
\begin{equation}\label{eq3.10}
 (I|_{S(c)})^{\prime}(v_n) \rightarrow 0 \quad \text{as} \enspace n \rightarrow \infty.
\end{equation}
Since $i(c) < 0$, it can be obtained that $\{ v_n \}$ is bounded in $W^{1,p}(\mathbb{R}^3)$. In fact, for $n$ sufficiently large, by Lemma \ref{lem2.1}, we have
\begin{equation}
\begin{split}
0 \geq I(v_n) &\geq \frac{a}{p}\left| \nabla v_n \right|_{p}^{p} + \frac{b}{2p}\left| \nabla v_n \right|_{p}^{2p} - \frac{c^{q-\frac{3(q-p)}{p}}}{p\left| Q \right|_{p}^{q-p}} \left| \nabla v_n \right|_{p}^{\frac{3(q-p)}{p}}\\
& \geq \frac{a}{p}\left| \nabla v_n \right|_{p}^{p} - \frac{c^{q-\frac{3(q-p)}{p}}}{p\left| Q \right|_{p}^{q-p}} \left| \nabla v_n \right|_{p}^{\frac{3(q-p)}{p}},\\
\end{split}
\nonumber
\end{equation}
which indicates $\left| \nabla v_n \right|_{p}$ is bounded in $W^{1,p}(\mathbb{R}^3)$ since there is $\frac{3(q-p)}{p} < p$ when $p<q<p+\frac{p^2}{3}$.
\\
Set
$$
\sigma := \underset{n \rightarrow +\infty}{\text{lim inf}} \underset{y \in \mathbb{R}^3}{\text{ sup}} \int_{B_1(y)}{|v_n|^pdx}.
$$
If $\sigma = 0$, then by vanishing lemma (see\cite{willem2012minimax}, Lemma 1.21), $u_n \rightarrow 0$ in $L^s(\mathbb{R}^3)$ for any $p<s<p^{\ast}$. Hence, $0 \leq \underset{n \rightarrow +\infty}{\text{lim}}\left(\frac{a}{p}\left| \nabla v_n \right|_{p}^{p} + \frac{b}{2p}\left| \nabla v_n \right|_{p}^{2p}\right) = i(c) < 0$, which is impossible. Therefore, we must have $\sigma > 0$, and then, there exists a sequence $\{ 
y_n \} \subset \mathbb{R}^3$ such that
$$
\int_{B_1(y_n)}{|v_n|^pdx} > 0.
$$
Set $u_n(\cdot) := v_n(\cdot + y_n)$, and then, $\{ u_n \}$ is still a bounded Palais-Smale sequence for $I|_{S(c)}$ on the level $i(c)$. We may assume that for some $u \in W^{1,p}(\mathbb{R}^3)$,
\begin{equation}\label{eq3.11}
\begin{cases}
u_n \rightharpoonup u \ne 0, & \text{in} \enspace W^{1,p}(\mathbb{R}^3),\\
u_n \rightarrow u, & \text{in} \enspace L_{loc}^{s}(\mathbb{R}^3),s \in [p,p^{\ast}),\\
u_n(x) \rightarrow u(x), & \text{a.e. in} \enspace \mathbb{R}^3.\\
\end{cases}
\end{equation}
Hence, $\alpha := |u|_{p} \in (0,c]$. Next, we try to prove  $u \in S(c)$ and suppose to the contrary that $\alpha \in (0,c)$.
\par
At first, we verify that, up to subsequence, $\nabla u_n(x) \rightarrow \nabla u(x)$ a.e. $x \in \mathbb{R}^3$. \\
We follow some ideas of \cite{du2023quasilinear} and \cite{luo2017existence}. By (\ref{eq3.10}) and Lemma \ref{lem2.4}, we have
\begin{equation}\label{eq3.12}
I^{\prime}(u_n) - \lambda_n \left| u_n \right|^{p-2}u_n 
\rightarrow 0 \quad \text{in} \quad \left(W^{1,p}(\mathbb{R}^3)\right)^{\prime} \quad as \quad n \rightarrow \infty. 
\end{equation}
where $\lambda_n = \frac{1}{c^p}\left< I^{\prime}(u_n) , u_n \right>$. Since $\{ u_n \}$ is bounded in $W^{1,p}(\mathbb{R}^3)$, we suppose that, up to subsequence,
\begin{equation}\label{eq3.13}
\underset{n \rightarrow +\infty}{\text{lim}}|\nabla u_n|_{p}^{p} = A \geq 0, \underset{n \rightarrow +\infty}{\text{lim}}\lambda_n = \lambda.
\end{equation}
Hence, by (\ref{eq3.12}), we have
\begin{equation}\label{eq3.14}
J(u_n) := -(a+bA)\varDelta_pu_n - |u_n|^{q-2}u_n - \lambda \left| u_n \right|^{p-2}u_n \rightarrow 0
\end{equation}
in $(W^{1,p}(\mathbb{R}^3))^{\prime}$ as $n \rightarrow \infty$.
Meanwhile, it is easy to see that
\begin{equation}\label{eq3.15}
J(u) = -(a+bA)\varDelta_pu - |u|^{q-2}u - \lambda \left| u \right|^{p-2}u \in \left(W^{1,p}(\mathbb{R}^3)\right)^{\prime}.
\end{equation}
Let $\eta \in C_{0}^{\infty}(\mathbb{R}^3,[0,1])$ satisfy $\eta|_{B_R} = 1$ and supp $\eta \subset B_{2R}$, where $B_R = \{ x\in \mathbb{R}^3 : |x| \leq R \}$. From (\ref{eq3.14}), (\ref{eq3.15}) and $(u_n - u)\eta \rightharpoonup 0$ in $W^{1,p}(\mathbb{R}^3)$, we obtain
\begin{equation}\label{eq3.16}
(J(u_n)-J(u))[(u_n - u)\eta] \rightarrow 0.
\end{equation}
By H\"{o}lder inequality and (\ref{eq3.11}), we conclude that
\begin{equation}\label{eq3.17}
\int_{\mathbb{R}^3}{(|u_n|^{s-2}u_n-|u|^{s-2}u)(u_n-u)\eta dx} = o_n(1), \forall p \leq s < p^{\ast},
\end{equation}
\begin{equation}\label{eq3.18}
\int_{\mathbb{R}^3}{(|\nabla u_n|^{p-2}\nabla u_n-|\nabla u|^{p-2}\nabla u)\nabla \eta(u_n-u) dx} = o_n(1).
\end{equation}
By (\ref{eq3.16})-(\ref{eq3.18}), we conclude that
$$
\int_{\mathbb{R}^3}{(|\nabla u_n|^{p-2}\nabla u_n-|\nabla u|^{p-2}\nabla u)\nabla(u_n-u)\eta dx} = o_n(1).
$$
Using (\ref{eq2.5}) and (\ref{eq2.6}), we infer that
$$
\underset{n \rightarrow \infty}{\text{lim}}\int_{B_{2R}}{|\nabla u_n - \nabla u|^pdx} = 0.
$$
Taking into account of the arbitrariness of $B_{2R}$, we conclude that
$$
\nabla u_n(x) \rightarrow \nabla u(x) \enspace a.e. \enspace x \in \mathbb{R}^3.
$$
Then, by the Brezis-Lieb Lemma (see\cite{willem2012minimax}, Lemma 1.32), we have
\begin{equation}\label{eq3.19}
|\nabla u_n |_{p}^{p} = |\nabla u_n - \nabla u |_{p}^{p} + |\nabla u |_{p}^{p} + o_n(1),
| u_n |_{p}^{p} = | u_n - u |_{p}^{p} + | u |_{p}^{p} + o_n(1).
\end{equation}
By (\ref{eq3.19}), it can be obtained that
$$
i(c) = \underset{n \rightarrow +\infty}{\text{lim}}I(u_n) \geq I(u) + \underset{n \rightarrow +\infty}{\text{lim}} I(u_n-u) \geq i(\alpha) + i(\sqrt[p]{c^p-\alpha^p}),
$$
which contradicts to Lemma \ref{lem3.4}. So, $|u|_{p} = c$ and $u_n \rightarrow u$ in $L^p(\mathbb{R}^3)$. We conclude from Lemma \ref{lem2.1} that $u_n \rightarrow u$ in $L^q(\mathbb{R}^3)$. Then, we have
$$
i(c) \leq I(u) \leq \underset{n \rightarrow +\infty}{\text{lim}} I(u_n) = i(c).
$$
Thus, $u \in S(c)$ is a minimizer of $i(c)$, and then, $u$ is a critical point of $I|_{S(c)}$. Therefore, by Lemma \ref{lem2.2}, there exists $\lambda < 0$ such that $(u,\lambda)$ is a couple of solution to (\ref{eq1.1}) and (\ref{eq1.2}).
\par
Secondly, we study the case when $q = p+\frac{p^2}{3}$:
\par
(2) For $q = p+\frac{p^2}{3}$, when $c> a^{\frac{3}{p^2}}\left| Q \right|_p$, we have $i(c) < 0$. Similar to the case when $p<q < p+\frac{p^2}{3}$, we can get a minimizer for $i(c)$ and thus get a couple of solution $(u, \lambda)$ to (\ref{eq1.1}) and (\ref{eq1.2}).
\par
When $0< c \leq a^{\frac{3}{p^2}}\left| Q \right|_p$, suppose that there exists a minimizer $u \in S(c)$ such that $I(u)=i(c)=0$, then
$$
\frac{a}{p}\left| \nabla u \right|_{p}^{p} + \frac{b}{2p}\left| \nabla u \right|_{p}^{2p} = \frac{1}{q}|u|_{q}^{q}
\leq
\frac{c^{\frac{p^2}{3}}}{p\left| Q \right|_{p}^{\frac{p^2}{3}}} \left| \nabla u \right|_{p}^{p}
\leq
\frac{a}{p}\left| \nabla u \right|_{p}^{p},
$$
which implies that $\left| \nabla u \right|_{p}^{2p}=0$. Hence $u=0$, which is impossible.
\par
Next, we study the case when $p+\frac{p^2}{3}<q < p+\frac{2p^2}{3}$:
\par
(3) When $c > c^{\ast}$, we have $i(c) < 0$. Similar to the case when $p<q < p+\frac{p^2}{3}$, we can get a minimizer for $i(c)$ and thus get a couple of solution $(u, \lambda)$ to (\ref{eq1.1}) and (\ref{eq1.2}).
\par
When $c < c^{\ast}$, suppose that there exists $u \in S(c)$ such that $I(u) = i(c) = 0$. Similar to the proof of Lemma \ref{lem3.4}, choose $\theta = \frac{c^{\ast}}{c}> 1$ and set $u^{\theta} = u(\theta^{-\frac{p}{3}}x)$. Then $u^{\theta} \in S(\theta c)$ and
$$
i(c^{\ast}) \leq I(u^{\theta}) < \theta^p I(u) = 0,
$$
which contradicts that $i(c^{\ast}) = 0$.
\par
When $c = c^{\ast}$, set $c_n = c^{\ast} + \frac{1}{n}$. For each $n$, there exsists $\{v_n \} \subset S(c)$ such that
\begin{equation}\label{eq3.20}
I(v_n) = i(c_n) < 0.
\end{equation}
By Lemma \ref{lem2.1}, we can deduce that $\{v_n \}$ is bounded in $W^{1,p}(\mathbb{R}^3)$. Moreover, by Lemma \ref{lem3.3}, we see that
$$
I(v_n) \rightarrow i(c^{\ast}) = 0.
$$
If $|v_n|_{q}^{q} \rightarrow 0$, then $\frac{a}{p}\left| \nabla v_n \right|_{p}^{p} + \frac{b}{2p}\left| \nabla v_n \right|_{p}^{2p} \rightarrow 0$. Hence, $\left| \nabla v_n \right|_{p}^{p} \rightarrow 0$, by Lemma \ref{lem2.1} we see that $I(v_n) \geq 0$ for $n$ large enough which contradicts to (\ref{eq3.20}). Therefore, similar to part (1), there exists $\{y_n\} \subset \mathbb{R}^3$ and $u \in W^{1,p}(\mathbb{R}^3)$such that
\begin{equation}
\begin{cases}
u_n \rightharpoonup u \ne 0, & \text{in} \enspace W^{1,p}(\mathbb{R}^3),\\
u_n \rightarrow u, & \text{in} \enspace L_{loc}^{s}(\mathbb{R}^3),s \in [p,p^{\ast}),\\
u_n(x) \rightarrow u(x), & \text{a.e. in} \enspace \mathbb{R}^3.\\
\end{cases}
\nonumber
\end{equation}
where $u_n(\cdot) := v_n(\cdot+y_n)$.
\par
We claim that $u \in S(c^{\ast})$ and suppose to the contrary that $|u|_{p} = \alpha \in (0,c^{\ast})$.
For each $n$, note that $I(u_n) = I(v_n) = i(c_n) < 0$, hence $u_n$ is also a minimizer for $i(c_n)$ and $\{u_n\}$ is bounded in $W^{1,p}(\mathbb{R}^3)$. By the Lagrange multipliers rule(\cite{chang2005methods}, Corollary 4.1.2), there exists $\{ \lambda _n \} \subset \mathbb{R}^3$ such that
$$
I^{\prime}(u_n) - \lambda_n \left| u_n \right|^{p-2}u_n = 0 \quad \text{in} \quad \left(W^{1,p}(\mathbb{R}^3)\right)^{\prime}.
$$
Obviously, $\{\lambda_n\}$ is bounded in $\mathbb{R}$ since $\{u_n\}$ is bounded in $W^{1,p}(\mathbb{R}^3)$.
Similar to (\ref{eq3.13}), suppose that, up to subsequence,
$$
\underset{n \rightarrow +\infty}{\text{lim}}|\nabla u_n|_{p}^{p} = A \geq 0, \underset{n \rightarrow +\infty}{\text{lim}}\lambda_n = \lambda.
$$
By a similar arugment in part (1), we obtain that $\nabla u_n(x) \rightarrow \nabla u(x)$ a.e. on $x \in \mathbb{R}^3$. Thus by the Brezis-Lieb Lemma, we have
$$
|\nabla u_n |_{p}^{p} = |\nabla u_n - \nabla u |_{p}^{p} + |\nabla u |_{p}^{p} + o_n(1),
| u_n |_{p}^{p} = | u_n - u |_{p}^{p} + | u |_{p}^{p} + o_n(1),
$$
moreover,
$$
0 = \underset{n \rightarrow +\infty}{\text{lim}}I(u_n) \geq I(u) + \underset{n \rightarrow +\infty}{\text{lim}}I(u_n - u).
$$
By Lemma \ref{lem3.1}, $i(\alpha) = i(\sqrt[p]{(c^{\ast})^p-\alpha^p}) = 0$. Because $\underset{n \rightarrow \infty}{\text{lim}}| u_n - u |_{p}^{p} =(c^{\ast})^p-\alpha^p$, we have
$$
\underset{n \rightarrow \infty}{\text{lim}}I(u_n-u) \geq i(\sqrt[p]{(c^{\ast})^p-\alpha^p}) =0,
$$
which implies $0 = I(u_n) \geq I(u) + I(u_n - u) \geq I(u) \geq i(\alpha) = 0 $. Hence $u$ is a minimizer of $i(\alpha)$, but, by Lemma \ref{lem3.1}, there is no minimizer for $\alpha \in (0,c^{\ast})$ . So we have $u \in S(c^{\ast})$ and $u_n \rightarrow u$ in $L^{p}(\mathbb{R}^3)$. Meanwhile, by Lemma \ref{lem2.1}, $u_n \rightarrow u$ in $L^{q}(\mathbb{R}^3)$. So we have
$$
0 = \underset{n \rightarrow \infty}{\text{lim}}I(u_n) \geq I(u) \geq i(c^{\ast}) = 0.
$$
We conclude that $u$ is a minimizer for $i(c^{\ast})$. Similar to part (1), there exists a $\lambda < 0$ such that $(u,\lambda)$ is a couple of solution to (\ref{eq1.1})-(\ref{eq1.2}).
\par
Finally, we study the case when $q = p+\frac{2p^2}{3}$:
\par
(4) When $c > \left(\frac{b\left| Q\right|_{p}^{\frac{2p^2}{3}}}{2}\right)^{\frac{3}{2p^2-3p}}$, by Lemma \ref{lem3.2}, we have $i(c) = -\infty$ so that $i(c)$ has no minimizer.
\par
When  $c \leq \left(\frac{b\left| Q\right|_{p}^{\frac{2p^2}{3}}}{2}\right)^{\frac{3}{2p^2-3p}}$, suppose that there exists a $u \in S(c)$ such that $I(u) = i(c) = 0$. Then
$$
\frac{a}{p}\left| \nabla u \right|_{p}^{p} + \frac{b}{2p}\left| \nabla u \right|_{p}^{2p} = \frac{1}{q}|u|_{q}^{q}
\leq \frac{c^{q-\frac{3(q-p)}{p}}}{p\left| Q \right|_{p}^{q-p}} \left| \nabla u \right|_{p}^{\frac{3(q-p)}{p}}
\leq \frac{b}{2p}\left| \nabla u \right|_{p}^{2p}.
$$
So $\left| \nabla u \right|_{p}^{p} = 0$, which is impossible.

\end{proof}

\section{Proof of Theorem 1.3}
In this section, we consider the minimization problem (\ref{eq1.7}) in the case $p+\frac{2p^2}{3} < q <p^{\ast}$.

\begin{lem}\label{lem4.1} Let $p+\frac{2p^2}{3} < q <p^{\ast}$, then $I$ is not bounded from below on $S_r(c)$ for any $c > 0$.
\end{lem}

\begin{proof}
For any $c>0$ and any $u \in S_r(c)$, set $u_t(x) := t^{\frac{3}{p}}u(tx)$, $t>0$, then $u_t \in S(c)$ and
$$
I(u_t) = t^{p}\frac{a}{p}\left| \nabla u \right|_{p}^{p} + t^{2p}\frac{b}{2p}\left| \nabla u \right|_{p}^{2p} - t^{\frac{3(q-p)}{p}}\frac{1}{q}\left| u \right|_{q}^{q}.
$$
Because $p+\frac{2p^2}{3} < q <p^{\ast}$, $\frac{3(q-p)}{p} > 2p$. Hence, $I(u_t) \rightarrow -\infty$ as $t \rightarrow +\infty$.
\end{proof}

\begin{lem}\label{lem4.2} Let $p+\frac{2p^2}{3} < q <p^{\ast}$, then $I$ is bounded from below and coercive on $M(c)$ for any $c > 0$. Moreover, there exists a constant $C_0 > 0$ such that $I(u) \geq C_0$ for all $u\in M(c)$.
\end{lem}
\begin{proof}
Since $P(u)=0$ for all $u \in M(c)$, we have
\begin{equation}\label{eq4.1}
\begin{split}
I(u) & = I(u) - \frac{p}{3(q-p)}P(u)\\
& = a\left(\frac{1}{p}- \frac{p}{3(q-p)}\right)\left| \nabla u \right|_{p}^{p} + b\left(\frac{1}{2p}-\frac{p}{3(q-p)}\right)\left| \nabla u \right|_{p}^{2p}.\\
\end{split}
\end{equation}
Note that $\frac{1}{2p}-\frac{p}{3(q-p)} > \frac{1}{p}- \frac{p}{3(q-p)} > 0$ when $p+\frac{2p^2}{3} < q$, we obtain that $I$ is coercive on $M(c)$.
\par
By Lemma \ref{lem2.1} and $P(u)=0$,
$$
b\left| \nabla u \right|_{p}^{2p} \leq a\left| \nabla u \right|_{p}^{p} + b\left| \nabla u \right|_{p}^{2p} = \frac{3(q-p)}{pq}\left| u \right|_{q}^{q}
\leq \frac{3(q-p)}{p} \frac{c^{q-\frac{3(q-p)}{p}}}{p\left| Q \right|_{p}^{q-p}} \left| \nabla u \right|_{p}^{\frac{3(q-p)}{p}}.
$$
Since $2p < \frac{3(q-p)}{p}$, we infer that $\left| \nabla u \right|_{p}$ has a positive lower bound. Hence, by (\ref{eq4.1}), there exists a constant $C_0 > 0$ such that $I(u) \geq C_0$ for all $u\in M(c)$.
\end{proof}

By Lemma \ref{lem4.2}, we know that $m(c) \geq C_0 > 0$.

\begin{lem}\label{lem4.3} 
 Let $p+\frac{2p^2}{3} < q <p^{\ast}$ and $c>0$, for any $u \in S_r(c)$ and $u_t(x) := t^{\frac{3}{p}}u(tx)$($t>0$), there exists a unique $t_0 > 0$ such that $I(u_{t_0}) = \underset{t>0}{max} I(u_t)$ and $u_{t_0} \in M(c)$. In particular, we have
$$
t_0 < 1  \Leftrightarrow P(u) < 0 ,t_0 = 1  \Leftrightarrow P(u) = 0.
$$
\end{lem}

\begin{proof}
For any $u \in S_r(c)$, let
$$
h(t) := I(u_t) = t^{p}\frac{a}{p}\left| \nabla u \right|_{p}^{p} + t^{2p}\frac{b}{2p}\left| \nabla u \right|_{p}^{2p} - t^{\frac{3(q-p)}{p}}\frac{1}{q}\left| u \right|_{q}^{q}, \enspace \forall t > 0.
$$
Differentiate $h(t)$ with respect to $t$, we obtain
$$
h^{\prime}(t) = \frac{t^{p}a\left| \nabla u \right|_{p}^{p} + t^{2p}b\left| \nabla u \right|_{p}^{2p} - t^{\frac{3(q-p)}{p}}\frac{3(q-p)}{qp}\left| u \right|_{q}^{q}}{t} = \frac{P(u_t)}{t}.
$$
Since $\frac{3(q-p)}{p} > 2p > p$, $h^{\prime}(t) > 0$ when $t > 0$ sufficiently small and $\underset{t \rightarrow +\infty}{\text{lim}}h^{\prime}(t) = -\infty$. Similar to \cite{li2021normalized}, we infer that $h(t)$ has a unique maximum at some point $t_0 > 0$. Therefore, $h^{\prime}(t_0) = \frac{P(u_{t_0})}{t_0} = 0$, which implies that $u_{t_0} \in M(c)$.
\par
Next, we prove the last two statements. We first claim that $P(u) < 0 \Rightarrow t_0 < 1$. Just suppose that $t_0 \geq 1$, by using $h^{\prime}(t_0) = 0$ and $P(u) < 0$, we obtain the following contradiction
\begin{equation}
\begin{split}
0 & = \frac{t_{0}^{p}a\left| \nabla u \right|_{p}^{p} + t_{0}^{2p}b\left| \nabla u \right|_{p}^{2p} - t_{0}^{\frac{3(q-p)}{p}}\frac{3(q-p)}{qp}\left| u \right|_{q}^{q}}{t_{0}^{\frac{3(q-p)}{p}}} \\
& \leq a\left| \nabla u \right|_{p}^{p}+b\left| \nabla u \right|_{p}^{2p} - \frac{3(q-p)}{qp}\left| u \right|_{q}^{q} < 0,
\end{split}
\nonumber
\end{equation}
So $t_0 < 1$. If $P(u) = 0$, then
\begin{equation}
\begin{split}
&\frac{t_{0}^{p}a\left| \nabla u \right|_{p}^{p} + t_{0}^{2p}b\left| \nabla u \right|_{p}^{2p} - t_{0}^{\frac{3(q-p)}{p}}\frac{3(q-p)}{qp}\left| u \right|_{q}^{q}}{t_{0}^{\frac{3(q-p)}{p}}} \\
& = 0\\
& = a\left| \nabla u \right|_{p}^{p}+b\left| \nabla u \right|_{p}^{2p} - \frac{3(q-p)}{qp}\left| u \right|_{q}^{q},
\end{split}
\nonumber
\end{equation}
which implies that neither $t_0 >1$ nor $t_0 <1$ could occur since $\left| \nabla u \right|_{p} \ne 0$. Thus $P(u)=0 \Rightarrow t_0 =1$. Next, we show that $t_0 < 1 \Rightarrow P(u) < 0$ and $t_0 = 1 \Rightarrow P(u) = 0$. If $t_0 <1$, then
\begin{equation}
\begin{split}
0 & = \frac{t_{0}^{p}a\left| \nabla u \right|_{p}^{p} + t_{0}^{2p}b\left| \nabla u \right|_{p}^{2p} - t_{0}^{\frac{3(q-p)}{p}}\frac{3(q-p)}{qp}\left| u \right|_{q}^{q}}{t_{0}^{\frac{3(q-p)}{p}}} \\
& > a\left| \nabla u \right|_{p}^{p}+b\left| \nabla u \right|_{p}^{2p} - \frac{3(q-p)}{qp}\left| u \right|_{q}^{q} = P(u).
\end{split}
\nonumber
\end{equation}
Also, $t_0=1$ implies $P(u)=P(u_{t_0}) = 0$.
\end{proof}

\begin{lem}\label{lem4.4} Let $c > 0$ and $p+\frac{2p^2}{3} < q <p^{\ast}$. Then $M(c)$ is a natural $C^1$-manifold and each minimizer of $I|_{M(c)}$ is a critical point of $I|_{S_r(c)}$.
\end{lem}
\begin{proof} For any $u \in M(c)$, we have $P(u)=0$. Set $Q(u) := |u|_{p}^{p}-c^p = 0$. Obviously, $P$, $Q$ are both of $C^1$ class. In order to prove that $M(c)$ is a manifold, we only need to verify that, for any $u \in M(c)$,
$$
\left(P^{\prime}(u),Q^{\prime}(u)\right) : W_{r}^{1,p}(\mathbb{R}^3) \rightarrow \mathbb{R}^2 \enspace \text{is} \enspace \text{surjection}.
$$
Assume by contradiction that $P^{\prime}(u)$ and $Q^{\prime}(u)$ are linearly dependent for some $u \in M(c)$, i.e., there exists a constant $l \in \mathbb{R}$, such that for any $\psi \in W_{r}^{1,p}(\mathbb{R}^3)$, we have $P^{\prime}(u)\psi = lQ^{\prime}(u)\psi$ that is
\begin{equation}\label{eq4.2}
-(a+2b|\nabla u|_{p}^{p})\varDelta _pu - \frac{3(q-p)}{p^2}|u|^{q-2}u = l |u|^{p-2}u.
\end{equation}
By the Nehari identity and Pohozaev identity of (\ref{eq4.2}), we obtain
$$
p^2\left(a|\nabla u|_{p}^{p} + 2b|\nabla u|_{p}^{2p}\right) = \frac{9(q-p)^2}{pq} |u|_{q}^{q},
$$
combining with $P(u)=0$, we have
$$
\left(\frac{3(q-p)}{p^2}-1\right)a|\nabla u|_{p}^{p} + \left(\frac{3(q-p)}{p^2}-2\right)b|\nabla u|_{p}^{2p} = 0.
$$
Since $\frac{3(q-p)}{p^2}-2 > 0$, we conclude that $|\nabla u|_{p}^{p} = 0$ which is impossible.
\par
Next, we prove the last statement. Suppose that $u$ is a minimizer of $I|_{M(c)}$, then $P(u) = 0$. By the Lagrange multiplier rule, there exist two Lagrange multipliers $\lambda$ and $\mu$ such that
\begin{equation}\label{eq4.3}
I^{\prime}(u) - \lambda |u|^{p-2}u - \mu P^{\prime}(u) = 0 \enspace in \enspace \left(W_{r}^{1,p}(\mathbb{R}^3)\right)^{\prime}.
\end{equation}
By the Nehari identity and Pohozaev identity of (\ref{eq4.3}), we obtain
$$
a|\nabla u|_{p}^{p} + b|\nabla u|_{p}^{2p} - \frac{3(q-p)}{qp}|u|_{q}^{q} - \mu\left[ap|\nabla u|_{p}^{p} + 2bp|\nabla u|_{p}^{2p} - \frac{3(q-p)^2}{qp^2}|u|_{q}^{q}\right] = 0.
$$
Recalling that $P(u) = a\left| \nabla u \right|_{p}^{p} + b\left| \nabla u \right|_{p}^{2p} - \frac{3(q-p)}{pq}\left| u \right|_{q}^{q} = 0$, we have
$$
\mu\left[ap|\nabla u|_{p}^{p} + 2bp|\nabla u|_{p}^{2p} - \frac{3(q-p)^2}{qp^2}|u|_{q}^{q}\right] = 0.
$$
Using $P(u)=0$ again, we obtain
$$
\mu\left[\left(p-\frac{3(q-p)}{p}\right)a|\nabla u|_{p}^{p} + \left(2p-\frac{3(q-p)}{p}\right)b|\nabla u|_{p}^{2p}\right] = 0.
$$
Since $p-\frac{3(q-p)}{p} < 2p-\frac{3(q-p)}{p} < 0$, we conclude that $\mu = 0$.
\end{proof}

\begin{lem}\label{lem4.5} For each $p+\frac{2p^2}{3} < q <p^{\ast}$, the function $c \mapsto m(c)$ is strictly decreasing for $c > 0$.
\end{lem}
\begin{proof} For any $0<c_1<c_2$, by $\gamma(c_1) > 0$ and Lemma \ref{lem4.3}, there exists $\{u_n\} \subset M(c_1)$ such that 
$$
I(u_n) = \underset{t>0}{\text{max}} I((u_n)_t) \leq m(c_1) + \frac{1}{n}.
$$
By (\ref{eq4.1}) and Lemma \ref{lem4.2}, there exists $k_i > 0(i=1,2,3,4)$ indenpendent of $n$ such that
\begin{equation}\label{eq4.4}
k_1 \leq |u_n|_{q}^{q} \leq k_2, k_3 \leq |\nabla u_n|_{p}^{p} \leq k_4.
\end{equation}
Let $v_n := \left(\frac{c_2}{c_1}\right)^{1-\frac{3}{p}}u_n\left(\frac{c_1}{c_2}x\right)$, then $v_n \in S_r(c_2)$ and
\begin{equation}\label{eq4.5}
|\nabla v_n|_{p}^{p} = |\nabla u_n|_{p}^{p}, |v_n|_{q}^{q} = 
\left(\frac{c_2}{c_1}\right)^{\frac{q(p-3)}{p}+3}|u_n|_{q}^{q}.
\end{equation}
Moreover, by Lemma \ref{lem4.3}, there exists $t_n > 0$ such that $(v_n)_{t_n} \in M(c_2)$ and $I((v_n)_{t_n}) = \underset{t>0}{\text{max}} I((v_n)_t)$. Since $\{(v_n)_{t_n}\} \subset M(c_2)$, by Lemma \ref{lem4.2}, $\{|\nabla (v_n)_{t_n}|_{p}^{p}\}$ has a positive lower bound which indenpendent of $n$. Combining with $|\nabla v_n|_{p}^{p} = |\nabla u_n|_{p}^{p}$, (\ref{eq4.4}) and $t_n^p = \frac{|\nabla (v_n)_{t_n}|_{p}^{p}}{|\nabla v_n|_{p}^{p}}$, we conclude that there exists a constant $C_1 > 0$ such that 
\begin{equation}\label{eq4.6}
t_n \geq C_1,
\end{equation}
for any $n$. 
Hence, by $q <p^{\ast} = \frac{3p}{3-p}$ and (\ref{eq4.4})-(\ref{eq4.6}), we have
\begin{equation}
\begin{split}
m(c_2) \leq I((v_n)_{t_n}) & = I((u_n)_{t_n}) - \left(\left(\frac{c_2}{c_1}\right)^{\frac{q(p-3)}{p}+3}-1\right)\frac{t_n^{\frac{3(q-p)}{p}}}{q} |u_n|_{q}^{q}\\
& \leq m(c_1)+\frac{1}{n}-\left(\left(\frac{c_2}{c_1}\right)^{\frac{q(p-3)}{p}+3}-1\right)\frac{C_1^{\frac{3(q-p)}{p}}}{q} k_1.\\
\end{split}
\end{equation}
which implies that $m(c_2) < m(c_1)$, by letting $n \rightarrow +\infty$.
\end{proof}

\begin{proof}[Proof of Theorem \ref{thm1.3}]
For $c > 0$, let $\{v_n\} \subset M(c)$ be a minimizing sequence for $m(c)$. If $|v_n|_{q}^{q} \rightarrow 0$, then by $P(v_n) = 0$, we deduce that $|\nabla v_n|_{p}^{p} \rightarrow 0$ which contradicts with Lemma \ref{lem4.2}. Thus, similar to the proof of Theorem \ref{thm1.1}, there exists $\{y_n\} \subset \mathbb{R}^3$ such that
$$
\int_{B_1(y_n)}{|v_n|^pdx} > 0.
$$
Denote $u_n(\cdot) := v_n(\cdot +y_n)$, then $\{u_n\} \subset M(c)$ is still a minimizing sequence for $m(c)$. By (\ref{eq4.1}), $\{u_n\}$ is bounded in $W_{r}^{1,p}(\mathbb{R}^3)$. Hence, up to subsequence, we may assume that there exists $u \in W_{r}^{1,p}(\mathbb{R}^3)$ such that
\begin{equation}\label{eq4.8}
\begin{cases}
u_n \rightharpoonup u \ne 0, & \text{in} \enspace W_{r}^{1,p}(\mathbb{R}^3),\\
u_n \rightarrow u, & \text{in} \enspace L^{s}(\mathbb{R}^3),s \in (p,p^{\ast}),\\
u_n(x) \rightarrow u(x), & \text{a.e. in} \enspace \mathbb{R}^3.\\
\end{cases}
\end{equation}
Then, $\alpha := |u|_{p} \in (0,c]$. Next, we show that $u$ is a minimizer of $I|_{M(c)}$ for $m(c)$.
By (\ref{eq4.8}) and $u_n \in M(c)$, we obtain that $P(u) \leq \underset{n \rightarrow +\infty}{\text{lim}}P(u_n) = 0$. Hence, by Lemma \ref{lem4.3}, there exists $t_0 \in (0,1]$ such that $u_{t_0} \in M(c)$. Thus, we have
\begin{equation}
\begin{split}
    m(\alpha) \leq I(u_{t_0}) 
& = I(u_{t_0}) - \frac{p}{3(q-p)}P(u_{t_0}) \\
& = t_{0}^{p}a\left(\frac{1}{p}- \frac{p}{3(q-p)}\right)|\nabla u |_{p}^{p} + t_{0}^{2p}b\left(\frac{1}{2p}-\frac{p}{3(q-p)}\right)| \nabla u|_{p}^{2p} \\
& \leq a\left(\frac{1}{p}- \frac{p}{3(q-p)}\right)| \nabla u |_{p}^{p} + b\left(\frac{1}{2p}-\frac{p}{3(q-p)}\right)|\nabla u |_{p}^{2p} \\
& \leq \underset{n \rightarrow +\infty}{\text{lim  inf }}\left[a\left(\frac{1}{p}- \frac{p}{3(q-p)}\right)| \nabla u_n |_{p}^{p} + b\left(\frac{1}{2p}-\frac{p}{3(q-p)}\right)| \nabla u_n |_{p}^{2p}\right] \\
& \leq \underset{n \rightarrow +\infty}{\text{lim inf }}\left(I(u_n)-\frac{p}{3(q-p)}P(u_n)\right) = m(c) \leq m(\alpha), \\
\end{split}
\nonumber
\end{equation}
which means $m(\alpha) = m(c)$ and $t_0 =1$. By Lemma \ref{lem4.5} and $\alpha \in (0,c]$, it must have $\alpha = c$. Thus, $u \in M(c)$ and $I(u) = m(c)$, i.e. $I|_{M(c)}$ attains its minimum at $u$. By Lemma \ref{lem4.4} and Lemma \ref{lem2.2}, we conclude that there exists $\lambda < 0$ such that $(u,\lambda) \in W_{r}^{1,p}(\mathbb{R}^3) \times \mathbb{R}$ is a couple of solution to (\ref{eq1.1})-(\ref{eq1.2}). Moreover, if $(v,\mu) \in W_{r}^{1,p}(\mathbb{R}^3) \times \mathbb{R}$ is also a couple of solution for $c$. Then, by Lemma \ref{lem2.2}, we have $P(v)=0$ wich implies that $v \in M(c)$. Since $I(u) = m(c) = \underset{w \in M(c)}{\text{inf}}I(w) \leq I(v)$, we see that $u$ is a radial ground state of (\ref{eq1.1})-(\ref{eq1.2}).
\end{proof}

\section{Proof of Theorem 1.4 and Theorem 1.5}
 In this section, we study the multiplicity and asymptotic behavior of normalized solutions to (\ref{eq1.1})-(\ref{eq1.2}) in the $L^p$ supercritical case. Fisrtly, we give some lemmas which are important for our proof of Theorem \ref{thm1.4}. We follow some ideas of \cite{zhang2022normalized}.
\par
Let $\{e_n^{\prime}\}_{n=1}^{\infty}$ be a Schauder basis of $W^{1,p}(\mathbb{R}^3)$(\cite{triebel1995interpolation}). Set
$$
e_n = \int_{O(N)}{e_n^{\prime}(g(x))d\mu_g},
$$
where $O(N)$ denotes the orthogonal group on $\mathbb{R}^3$ and $d\mu_g$ is the Haar measure on $O(N)$. Then going if necessary to select one in identical elements, we see that $\{e_n\}_{n=1}^{\infty}$ is a Schauder basis of $W_{r}^{1,p}(\mathbb{R}^3)$. Without loss of generality, we assume that $\vert e_n \vert = 1$ for any $n \geq 1$. $\forall n \in \mathbb{N}^+$, denote
$$
V_n := \text{span} \{ e_1, ... , e_n\},V_n^{\bot} := \overline{\text{span}\{e_i : i \geq n+1\}}.
$$
Clearly, $W_{r}^{1,p}(\mathbb{R}^3) = V_n \oplus V_n^{\bot}$ for any $n \in \mathbb{N}^+$.

\begin{lem}\label{lem5.1} (\cite{zhang2022normalized}, Lemma 6.1) Let $p+\frac{2p^2}{3} < q <p^{\ast}$ and $c > 0$, then there holds
$$
\mu_n := \underset{u \in V_n^{\bot}\cap S_r(c)}{\text{inf}} \frac{|\nabla u|_p^p}{|u|_q^p} \rightarrow +\infty, \enspace \text{as} \enspace n \rightarrow +\infty.
$$
\end{lem}

For $n \in \mathbb{N}^+$, we define
$$
\rho_n := \left(\frac{a(\mu_n)^{\frac{q}{p}}}{L}\right)^{\frac{1}{q-p}},
$$
where
$$
L := \underset{x > 0}{\text{max}} \frac{(x^p+1)^{\frac{q}{p}}}{x^q+1}.
$$
Let $B_n := \{u \in V_n^{\bot}\cap S_r(c) : |\nabla u|_p = \rho_n\}$. We have
\begin{lem}\label{lem5.2}
$$
\beta_n := \underset{u \in B_n}{\text{inf}} I(u) \rightarrow +\infty, \enspace \text{as} \enspace n \rightarrow +\infty.
$$
\end{lem}

\begin{proof}
$\forall u \in B_n$, by Lemma \ref{lem5.1},
$$
|u|_q^p \leq \mu_n^{-1}|\nabla u|_p^p \leq \mu_n^{-1}(|\nabla u|_p^p+1).
$$
Then,
\begin{equation}
\begin{split}
I(u) & = \frac{a}{p}|\nabla u|_p^p + \frac{b}{2p}|\nabla u|_p^{2p}-\frac{1}{q}|u|_q^q \\
& \geq \frac{a}{p}|\nabla u|_p^p- \frac{1}{q(\mu_n)^{\frac{q}{p}}}\left(|\nabla u|_p^p+1\right)^{\frac{q}{p}}\\
& \geq \frac{a}{p}|\nabla u|_p^p- \frac{L}{q(\mu_n)^{\frac{q}{p}}}\left(|\nabla u|_p^q+1\right)\\
& = \frac{a}{p}(\rho_n)^p-\frac{L}{q(\mu_n)^{\frac{q}{p}}}(\rho_n)^q - \frac{L}{q(\mu_n)^{\frac{q}{p}}}\\
& = \left(\frac{1}{p}-\frac{1}{q}\right)a(\rho_n)^p - \frac{L}{q(\mu_n)^{\frac{q}{p}}}.
\end{split}
\nonumber
\end{equation}
Thus $\beta_n \geq \left(\frac{1}{p}-\frac{1}{q}\right)a(\rho_n)^p - \frac{L}{q(\mu_n)^{\frac{q}{p}}}$. By Lemma \ref{lem5.1} and definition of $\rho_n$, we obtain that $\beta_n \rightarrow +\infty$ as $n \rightarrow +\infty$.
\end{proof}

\begin{lem}\label{lem5.3} There exits $0< \rho_0^{\prime} < \rho_0$ such that
$$
\underset{u \in S_r(c),|\nabla u|_p \leq \rho_0^{\prime}}{\text{sup}} I(u) < \beta_0 := \underset{u \in S_r(c),|\nabla u|_p = \rho_0}{\text{inf}} I(u).
$$
Moreover, $\beta_0 > 0$.
\end{lem}
\begin{proof}
Since $u \in S_r(c)$, by Lemma \ref{lem2.1}, we have
\begin{equation}
\begin{split}
& \frac{a}{p}\left| \nabla u \right|_{p}^{p} + \frac{b}{2p}\left| \nabla u \right|_{p}^{2p} - \frac{c^{q-\frac{3(q-p)}{p}}}{p\left| Q \right|_{p}^{q-p}} \left| \nabla u \right|_{p}^{\frac{3(q-p)}{p}}\\
& \leq I(u) \leq \frac{a}{p}\left| \nabla u \right|_{p}^{p} + \frac{b}{2p}\left| \nabla u \right|_{p}^{2p}.
\end{split}
\nonumber
\end{equation}
Note that $\frac{3(q-p)}{p} > 2p$, we conclude that there exits $\rho_0 > 0$ small enough such that $\beta_0 > 0$. Meanwhile, we can choose $0< \rho_0^{\prime} < \rho_0$ small enough such that
$$
\underset{u \in S_r(c),|\nabla u|_p \leq \rho_0^{\prime}}{\text{sup}} I(u) < \frac{\beta_0}{2} < \beta_0.
$$
\end{proof}

By Lemma \ref{lem5.2}, without loss of generality, we assume that $\beta_n > \beta_0$ for any $n \in \mathbb{N}^+$.

Next we begin to set up our min-max procedure. First we introduce the map
\begin{equation}\label{eq5.1}
\begin{split}
k: W_r^{1,p}(\mathbb{R}^3) \times \mathbb{R} & \rightarrow W_r^{1,p}(\mathbb{R}^3) \\
(u,\theta) & \rightarrow k(u,\theta):= e^{\frac{3}{p}\theta}u(e^{\theta}x).\\
 \end{split}
 \end{equation}
Then for any given $u \in S_r(c)$, we have $k(u,\theta) \in S_r(c)$ and
\begin{equation}\label{eq5.2}
|\nabla k(u,\theta)|_p^p = e^{p\theta}|\nabla u|_p^p,
|k(u,\theta)|_q^q = e^{\frac{3(q-p)}{p}\theta} |u|_q^q,
\end{equation}
and 
\begin{equation}\label{eq5.3}
I(k(u,\theta)) = e^{p\theta}\frac{a}{p}|\nabla u|_p^p + e^{2p\theta}\frac{b}{2p}|\nabla u|_p^{2p} - e^{\frac{3(q-p)}{p}\theta} \frac{1}{q}|u|_q^q.
\end{equation}
Note that $\frac{3(q-p)}{p} > 2p$, we conclude from (\ref{eq5.1})-(\ref{eq5.3}) that
\begin{equation}
\begin{cases}
|\nabla k(u,\theta)|_p^p \rightarrow 0,I(k(u,\theta)) \rightarrow 0, & \theta \rightarrow -\infty,\\
|\nabla k(u,\theta)|_p^p \rightarrow +\infty,I(k(u,\theta)) \rightarrow -\infty, & \theta \rightarrow +\infty.\\
\end{cases}
\nonumber
\end{equation}

Since $V_n$ is finite dimensional, it follows that for each $n \in \mathbb{N}^+$, there exists $\theta_n > 0$ such that
\begin{equation}\label{eq5.4}
\overline{g}_n \in C([0,1] \times (S_r(c)\cap V_n),S_r(c)): \overline{g}_n(t,u) = k(u,(2t-1)\theta_n)
\end{equation}
satisfies
\begin{equation}\label{eq5.5}
\begin{cases}
|\nabla \overline{g}_n(0,u)|_p^p < (\rho_0^{\prime})^p < (\rho_n)^p, &|\nabla \overline{g}_n(1,u)|_p^p > (\rho_n)^p,\\
I(\overline{g}_n(0,u)) < \text{max}\{ \beta_0, \beta_n\}, &I(\overline{g}_n(1,u)) < \beta_n,\\
\end{cases}
\end{equation}
for all $u \in S_r(c)\cap V_n$.
Now we define
\begin{equation}\label{5.6}
\begin{split}
\Gamma_n := & \{g \in C([0,1]\times (S_r(c)\cap V_n),S_r(c)): g(t,-u)=-g(t,u), \\
& g(0,u)= \overline{g}_n(0,u),g(1,u)= \overline{g}_n(1,u), \forall u \in (S_r(c)\cap V_n)\}.\\
\end{split}
\end{equation}
Clearly $\overline{g}_n \in \Gamma_n$. We introduce the following linking property:

\begin{lem}\label{lem5.4} (\cite{bartsch2012normalized}, Lemma 2.3) For each $g \in \Gamma_n$, there exists $(t,u) \in [0,1]\times (S_r(c)\cap V_n)$ such that $g(t,u) \in B_n$.
\end{lem}

Following Lemma \ref{lem5.4}, we have
\begin{lem}\label{lem5.5} For each $n \in \mathbb{N}^+$,
$$
\gamma_n(c) := \underset{g \in \Gamma_n}{\text{inf}}\underset{(t,u) \in [0,1]\times (S_r(c)\cap V_n)}{\text{max}} I(g(t,u)) \geq \beta_n.
$$
\end{lem}

On the other hand, $\forall g \in \Gamma_n$,$\forall u \in S_r(c)\cap V_n$, by (\ref{eq5.5}), we have
$$
|\nabla g(0,u)|_p = |\nabla \overline{g}_n(0,u)|_p <\rho_0^{\prime} < \rho_n,
$$
$$
|\nabla g(1,u)|_p= |\nabla \overline{g}_n(1,u)|_p > \rho_n >\rho_0.
$$
By the continuity of $|\nabla g(t,u)|_p$ with respect to $t$, there exits $t \in (0,1)$ such that $|\nabla g(t,u)|_p = \rho_0$. Then $\underset{(t,u) \in [0,1]\times (S_r(c)\cap V_n)}{\text{max}} I(g(t,u)) \geq \beta_0$. Since $g \in \Gamma_n$ is arbitrary, we deduce that
\begin{equation}\label{eq5.7}
\gamma_n(c) \geq \beta_0 > 0.
\end{equation}

Fix an arbitrary $n \in \mathbb{N}^+$ from now on. We adopt the approach developed by Jeanjean\cite{jeanjean1997existence}. First, we introduce the auxiliary functional
$$
\tilde{I} : S_r(c) \times \mathbb{R} \rightarrow \mathbb{R}, (u,\theta) \rightarrow I(k(u,\theta)),
$$
where $k(u,\theta)$ is given in (\ref{eq5.1}). Set
\begin{equation}\label{5.8}
\begin{split}
\tilde{\Gamma}_n := & \{\tilde{g} \in C([0,1]\times (S_r(c)\cap V_n),S_r(c)\times \mathbb{R}): \tilde{g}(t,-u)=-\tilde{g}(t,u), \\
& k(\tilde{g}(0,u))= \overline{g}_n(0,u),k(\tilde{g}(1,u))= \overline{g}_n(1,u), \forall u \in (S_r(c)\cap V_n)\},\\
\end{split}
\end{equation}
and
$$
\tilde{\gamma}_n(c) := \underset{\tilde{g} \in \tilde{\Gamma}_n}{\text{inf}}\underset{(t,u) \in [0,1]\times (S_r(c)\cap V_n)}{\text{max}} \tilde{I}(\tilde{g}(t,u)),
$$
we have
\begin{lem}\label{lem5.6}
$$
\tilde{\gamma}_n(c) = \gamma_n(c).
$$
\end{lem}
\begin{proof}
It follows immediately from the definition of $\tilde{\gamma}_n(c)$ and $\gamma_n(c)$ along with the fact that the maps
$$
\varphi: \Gamma_n \rightarrow \tilde{\Gamma}_n, g \rightarrow \varphi(g) := (g,0)
$$
and
$$
\psi :\tilde{\Gamma}_n \rightarrow \Gamma_n, \tilde{g} \rightarrow \psi(\tilde{g}) := k\circ \tilde{g}
$$
satisfy
$$
\tilde{I}(\varphi(g)) = I(g) \enspace and \enspace I(\psi(\tilde{g})) = \tilde{I}(\tilde{g}).
$$
\end{proof}

Therefore, it follows from (\ref{eq5.7}), Lemma \ref{lem5.5} and Lemma \ref{lem5.6} that
\begin{equation}
\begin{split}
& \tilde{\gamma}_n(c) = \gamma_n(c) \\
&\geq \text{max}\{ \beta_n,\beta_0 \}\\
& > \text{max}\left\{\underset{(t,u) \in [0,1]\times (S_r(c)\cap V_n)}{\text{max}} I(\overline{g}_n(0,u)),\underset{(t,u) \in [0,1]\times (S_r(c)\cap V_n)}{\text{max}} I(\overline{g}_n(1,u))\right\}.\\
\end{split}
\nonumber
\end{equation}

\begin{lem}\label{lem5.7} There exists a Palais-Smale sequence $\{u_k\} \subset S_r(c)$ for $I|_{S_r(c)}$ at level $\gamma_n(c)$ such that
$$
P(u_k) \rightarrow 0, \enspace as \enspace k \rightarrow +\infty.
$$
\end{lem}
\begin{proof}
Our proof is similar to the Lemma 6.8 of \cite{zhang2022normalized},  so we omit it.
\end{proof}

\begin{lem}\label{lem5.8} Let $\{u_k\} \subset S_r(c)$ be the Palais-Smale sequence obtained in Lemma \ref{lem5.7}. Then there exist $u \in W_r^{1,p}(\mathbb{R}^3)$ and $\{\lambda_k\} \subset \mathbb{R}$ such that up to a subsequence, as $k \rightarrow +\infty$
\\(i) $u_k \rightharpoonup u \ne 0$ in $W_r^{1,p}(\mathbb{R}^3)$;
\\(ii) $\lambda_k \rightarrow \lambda < 0$ in $\mathbb{R}$;
\\(iii) $|\nabla u_k|_p \rightarrow |\nabla u|_p$;
\\(iv) $-(a+b|\nabla u_k|_p^p)\varDelta_pu_k - |u_k|^{q-2}u_k - \lambda_k \left| u_k \right|^{p-2}u_k \rightarrow 0$ in $\left(W_r^{1,p}(\mathbb{R}^3)\right)^{\prime}$;
\\(v) $-(a+b|\nabla u|_p^p)\varDelta_pu - |u|^{q-2}u - \lambda \left| u \right|^{p-2}u = 0$ in $\left(W_r^{1,p}(\mathbb{R}^3)\right)^{\prime}$;
\\(vi) $u_k \rightarrow u$ in $W_r^{1,p}(\mathbb{R}^3)$.
\end{lem}

\begin{proof}
Since $I(u_k) \rightarrow \gamma_n(c)$ and $P(u_k) \rightarrow 0$ as $k \rightarrow +\infty$, we have
\begin{equation}
\begin{split}
\gamma_n(c) &=I(u_k) + o(1) \\
& = I(u_k) - \frac{p}{3(q-p)}P(u_k) + o(1) \\
& = a\left(\frac{1}{p}- \frac{p}{3(q-p)}\right)\left| \nabla u_k \right|_{p}^{p} + b\left(\frac{1}{2p}-\frac{p}{3(q-p)}\right)\left| \nabla u_k \right|_{p}^{2p} + o(1).\\
\end{split}
\nonumber
\end{equation}
Thus $\{u_k\}$ is bounded in $W_r^{1,p}(\mathbb{R}^3)$. So, up to a subsequence, there exist $u \in W_r^{1,p}(\mathbb{R}^3)$ such that
\begin{equation}\label{eq5.9}
\begin{cases}
u_k \rightharpoonup u, & \text{in} \enspace W_{r}^{1,p}(\mathbb{R}^3),\\
u_k \rightarrow u, & \text{in} \enspace L^{q}(\mathbb{R}^3),\\
u_k(x) \rightarrow u(x), & \text{a.e. in} \enspace \mathbb{R}^3.\\
\end{cases}
\end{equation}
If $u = 0$, we have $|u_k|_q^q = o(1)$. Thus we obtain $|\nabla u_k|_p^p = o(1)$ because $P(u_k) = o(1)$. As a consequence, $I(u_k) = o(1)$, which contradicts with (\ref{eq5.7}). Thus (i) is obtained. By Lemma \ref{lem2.4}, we  have
$$
-(a+b|\nabla u_k|_p^p)\varDelta_pu_k - |u_k|^{q-2}u_k - \lambda_k \left| u_k \right|^{p-2}u_k \rightarrow 0
$$
in $\left(W_r^{1,p}(\mathbb{R}^3)\right)^{\prime}$ as $k \rightarrow +\infty$, where
\begin{equation}\label{eq5.10}
\lambda_k = \frac{1}{c^p}\left(a|\nabla u_k|_p^p+b|\nabla u_k|_p^{2p}-|u_k|_q^q\right),
\end{equation}
then (iv) is proved. Since $\{u_k\}$ is bounded in $W_r^{1,p}(\mathbb{R}^3)$, there exits $\lambda \in \mathbb{R}$ such that $\lambda_k \rightarrow \lambda$ as $k \rightarrow +\infty$ up to a subsequence. Furthermore,
\begin{equation}\label{eq5.11}
\begin{split}
\lambda & = \underset{k \rightarrow +\infty}{\text{lim}}\lambda_k\\
& = \underset{k \rightarrow +\infty}{\text{lim}} \frac{1}{c^p}\left(a|\nabla u_k|_p^p+b|\nabla u_k|_p^{2p}-|u_k|_q^q\right)\\
& \leq \underset{k \rightarrow +\infty}{\text{lim}} \frac{1}{c^p}\left(a|\nabla u_k|_p^p+b|\nabla u_k|_p^{2p}-\frac{3(q-p)}{qp}|u_k|_q^q\right)\\
& = \underset{k \rightarrow +\infty}{\text{lim}} \frac{1}{c^p}P(u_k)\\
& =0.\\
\end{split}
\end{equation}
Since $|u_k|_q^q \rightarrow |u|_q^q \ne 0$, we obtain $\lambda < 0$. Thus (ii) is proved. Next we prove that (iii) holds. Suppose that $\underset{k \rightarrow +\infty}{\text{lim}}|\nabla u_k|_p^p = A \geq 0$. Then, from (ii) and (iv), we have
\begin{equation}\label{eq5.12}
-(a+bA)\varDelta_pu_k - |u_k|^{q-2}u_k - \lambda \left| u_k \right|^{p-2}u_k \rightarrow 0 \enspace in \enspace \left(W_r^{1,p}(\mathbb{R}^3)\right)^{\prime}, \enspace as \enspace k \rightarrow +\infty.
\end{equation}
Similar to the Proof of Theorem \ref{thm1.1}, we can obtain that 
\begin{equation}\label{eq5.13}
\nabla u_k \rightarrow \nabla u\enspace \text{a.e.} \enspace \text{on} \enspace \mathbb{R}^3.
\end{equation}
In view of the uniqueness of weak limit, using Proposition 5.4.7 of \cite{willem2023functional}, by (\ref{eq5.9}), (\ref{eq5.12}) and (\ref{eq5.13}), we have
\begin{equation}\label{eq5.14}
-(a+bA)\varDelta_pu - |u|^{q-2}u - \lambda \left| u \right|^{p-2}u = 0 \enspace in \enspace (W_r^{1,p}(\mathbb{R}^3))^{\prime}.
\end{equation}
From (\ref{eq5.12}) and (\ref{eq5.14}), we have
\begin{equation}\label{eq5.15}
(a+bA)|\nabla u_k|_p^p - |u_k|_q^q - \lambda | u_k |_p^p = o(1),
\end{equation}
and
\begin{equation}\label{eq5.16}
(a+bA)|\nabla u|_p^p - |u|_q^q - \lambda | u |_p^p = 0.
\end{equation}
Combining (\ref{eq5.15}) with (\ref{eq5.16}), we obtain
$$
(a+bA)(|\nabla u_k|_p^p - |\nabla u|_p^p) - \lambda (|u_k|_p^p-|u|_p^p)= (|u_k|_q^q-|u|_q^q)+o(1).
$$
Since $u_k \rightarrow u$ in $L^q(\mathbb{R}^3)$,
$$
(a+bA)(|\nabla u_k|_p^p - |\nabla u|_p^p) - \lambda (|u_k|_p^p-|u|_p^p)= o(1).
$$
Since $\lambda < 0$, by the weakly lower semicontinuity of norm, we obtain 
\begin{equation}\label{eq5.17}
|\nabla u_k|_p^p - |\nabla u|_p^p = o(1), |u_k|_p^p-|u|_p^p = o(1).
\end{equation}
Thus (iii) is proved and (v) is easily deduced by (\ref{eq5.14}) and (\ref{eq5.17}). Finally, by (\ref{eq5.9}) and (\ref{eq5.17}), we have $u_k \rightarrow u$ in $W_r^{1,p}(\mathbb{R}^3)$ as $k \rightarrow +\infty$.
\end{proof}

\begin{proof}[Proof of Theorem \ref{thm1.4}]
By Lemma \ref{lem5.8}, we get that for any fixed $c > 0$, (\ref{eq1.1})-(\ref{eq1.2}) has a sequence of couples of normalized solutions $\{(u_n,\lambda_n)\} \subset S_r(c) \times \mathbb{R}^-$ with $I(u_n) = \gamma_n(c)$ for each $n \in \mathbb{N}^+$. By Lemma \ref{lem2.2}, we know that $P(u_n)=0$ for each $n \in \mathbb{N}^+$. Thus we have
\begin{equation}
\begin{split}
\gamma_n(c) &=I(u_n) = I(u_n) - \frac{p}{3(q-p)}P(u_n)\\
& = a\left(\frac{1}{p}- \frac{p}{3(q-p)}\right)\left| \nabla u_n \right|_{p}^{p} + b\left(\frac{1}{2p}-\frac{p}{3(q-p)}\right)\left| \nabla u_n \right|_{p}^{2p}.\\
\end{split}
\nonumber
\end{equation}
Since $\gamma_n(c) \geq \beta_n \rightarrow +\infty$ as $n \rightarrow +\infty$, we conclude that $\Vert u_n \Vert_{ W_{r}^{1,p}(\mathbb{R}^3) } \rightarrow +\infty$ and $I(u_n) \rightarrow +\infty$ as $n \rightarrow +\infty$.
\end{proof}

Next, we study the asymptotic behavior of solutions obtained in Theorem \ref{thm1.4} when $b \rightarrow 0^+$. Fix a $c > 0$ from now on. Since $b$ is a variable in this case, we replace $I$, $P$, $\gamma_n(c)$ by $I_b$, $P_b$, $\gamma_n^b(c)$, respectively.  For $b > 0$, we denote $\{(u_n^b,\lambda_n^b)\} \subset S_r(c) \times \mathbb{R}^-$ as a sequence of couples of weak solutions which we obtained in Theorem \ref{thm1.4}. Then, it follows that, for any $n \in \mathbb{N}^+$,
\begin{equation}\label{eq5.18}
I_b(u_n^b) = \gamma_n^b(c).
\end{equation}
Recalling the proof of Lemma \ref{lem5.2}, we see that $\beta_n$ is independent of $b$. Thus, we have
\begin{equation}\label{eq5.19}
\gamma_n^b(c) \geq \beta_n > 0,
\end{equation}
for any $b > 0$.

\begin{proof}[Proof of Theorem \ref{thm1.5}]
Firstly, we claim that for any $\{b_m\} \rightarrow 0^+(m \rightarrow +\infty)$, $\{(u_n^{b_m}\}_{m \in \mathbb{N}^+} \subset S_r(c)$ is bounded in $W_r^{1,p}(\mathbb{R}^3)$. Without loss of generality, we suppose that $b_m \leq 1$ for any $m \in \mathbb{N}^+$. Since $V_n$ is finite-dimensional, we obtain that for each $n \in \mathbb{N}^+$,
\begin{equation}
\begin{split}
\gamma_n^{b_m}(c) &:= \underset{g \in \Gamma_n}{\text{inf}}\underset{(t,u) \in [0,1]\times (S_r(c)\cap V_n)}{\text{max}} I_{b_m}(g(t,u)) \\
& \leq \underset{g \in \Gamma_n}{\text{inf}}\underset{(t,u) \in [0,1]\times (S_r(c)\cap V_n)}{\text{max}} I_1(g(t,u)) := \alpha_n < +\infty. \\
\end{split}
\nonumber
\end{equation}
Since $\{(u_n^{b_m},\lambda_n^{b_m})\} \subset S_r(c) \times \mathbb{R}^-$ is a sequence of couples of weak solutions to (\ref{eq1.1})-(\ref{eq1.2}) with $b=b_m$ and
\begin{equation}\label{eq5.20}
\lambda_n^{b_m} = \frac{1}{c^p}\left(a|\nabla u_n^{b_m}|_p^p+b_m|\nabla u_n^{b_m}|_p^{2p}-|u_n^{b_m}|_q^q\right).
\end{equation}
By Lemma \ref{lem2.3}, we deduce that $P_{b_m}(u_n^{b_m}) = 0$. Then
\begin{equation}
\begin{split}
\gamma_n^{b_m}(c) = I_{b_m}(u_n^{b_m}) & = I_{b_m}(u_n^{b_m}) - \frac{p}{3(q-p)}P_{b_m}(u_n^{b_m})\\
& = a\left(\frac{1}{p}- \frac{p}{3(q-p)}\right)\left| \nabla u_n^{b_m} \right|_{p}^{p} + b\left(\frac{1}{2p}-\frac{p}{3(q-p)}\right)\left| \nabla u_n^{b_m} \right|_{p}^{2p},\\
\end{split}
\nonumber
\end{equation}
we conclude from $\gamma_n^{b_m}(c) < +\infty$ that $\{(u_n^{b_m})\}_{m \in \mathbb{N}^+}$ is bounded in $W_r^{1,p}(\mathbb{R}^3)$. Combining $P_{b_m}(u_n^{b_m}) = 0$ with (\ref{eq5.20}), $\{\lambda_n^{b_m}\}$ is bounded in $\mathbb{R}$.
Then there exists a subsequence of $\{b_m\}$, still denoted by $\{b_m\}$, and $\lambda_n^0 \leq 0$ such that as $m \rightarrow +\infty$, $\lambda_n^{b_m} \rightarrow \lambda_n^0$ and
\begin{equation}\label{eq5.21}
\begin{cases}
u_n^{b_m} \rightharpoonup u_n^0, & \text{in} \enspace W_{r}^{1,p}(\mathbb{R}^3),\\
u_n^{b_m} \rightarrow u_n^0, & \text{in} \enspace L^{q}(\mathbb{R}^3),\\
u_n^{b_m}(x) \rightarrow u_n^0(x), & \text{a.e. in} \enspace \mathbb{R}^3.\\
\end{cases}
\end{equation}
Since $\{(u_n^{b_m},\lambda_n^{b_m})\}$ is a sequence of couples of weak solutions to (\ref{eq1.1})-(\ref{eq1.2}), by $b_m \rightarrow 0$ and $\lambda_n^{b_m} \rightarrow \lambda_n^0$ as $m \rightarrow +\infty$, we have
\begin{equation}\label{eq5.22}
-a\varDelta_p u_n^{b_m} - |u_n^{b_m}|^{q-2}u_n^{b_m} - \lambda_n^0|u_n^{b_m}|^{p-2}u_n^{b_m} = o_m(1).
\end{equation}
Similar to the proof of Theorem \ref{thm1.1}, we can get, as $m \rightarrow +\infty$,
\begin{equation}\label{eq5.23}
\nabla u_n^{b_m}(x) \rightarrow \nabla u_n^0(x), \text{a.e.in} \enspace \mathbb{R}^3.
\end{equation}
In view of the uniqueness of weak limit, using Proposition 5.4.7 of \cite{willem2023functional}, by (\ref{eq5.21}), (\ref{eq5.22}) and (\ref{eq5.23}), we have
\begin{equation}\label{eq5.24}
-a\varDelta_p u_n^0 - |u_n^0|^{q-2}u_n^0 - \lambda_n^0|u_n^0|^{p-2}u_n^0 = 0.
\end{equation}
By (\ref{eq5.22}) and (\ref{eq5.24}), we have
$$
a|\nabla u_n^{b_m}|_p^p - |u_n^{b_m}|_q^q - \lambda_n^0|u_n^{b_m}|_p^p = o_m(1),
$$
and
$$
a|\nabla u_n^0|_p^p - |u_n^0|_q^q - \lambda_n^0|u_n^0|_p^p = 0.
$$
Thus we obtain
$$
a(|\nabla u_n^{b_m}|_p^p - |\nabla u_n^0|_p^p)- \lambda_n^0(|u_n^{b_m}|_p^p - |u_n^0|_p^p) = |u_n^{b_m}|_q^q - |u_n^0|_q^q + o_m(1),
$$
by \ref{eq5.21}, we have
\begin{equation}\label{eq5.25}
a(|\nabla u_n^{b_m}|_p^p - |\nabla u_n^0|_p^p)- \lambda_n^0(|u_n^{b_m}|_p^p - |u_n^0|_p^p) = o_m(1).
\end{equation}
Since  $\lambda_n^0 \leq 0$, by the weakly lower semicontinuity of norm, we have 
\begin{equation}\label{eq5.26}
|\nabla u_n^{b_m}|_p^p - |\nabla u_n^0|_p^p = o_m(1).
\end{equation}
If $\lambda_n^0 = 0$, by (\ref{eq5.24}), $u_n^0$ is a weak solution to $-a\varDelta_p u_n^0 = |u_n^0|^{q-2}u_n^0$. By the Nehari and Pohozaev identities, we infer that $u_n^0 = 0$. Then, by (\ref{eq5.21}), (\ref{eq5.26}), as $m \rightarrow +\infty$, we have
$$
0 < \beta_n \leq \gamma_n^{b_m}(c) = I_{b_m}(u_n^{b_m}) \rightarrow I_0(u_n^0) = 0,
$$
which is a contradiction. Thus $\lambda_n^0 < 0$ and according to (\ref{eq5.25}), we have $|u_n^{b_m}|_p^p - |u_n^0|_p^p = o_m(1)$ so that $u_n^0 \in S_r(c)$. Together with (\ref{eq5.21}) and (\ref{eq5.26}), we conclude that $u_n^{b_m} \rightarrow u_n^0$ in $W_r^{1,p}(\mathbb{R}^3)$ as $m \rightarrow +\infty$. Moreover, by (\ref{eq5.24}), $\{u_n^0, \lambda_n^0\} \subset S_r(c) \times \mathbb{R}^-$ is a sequence of couples of weak solutions to the following equation
$$
-a\varDelta_p u - \lambda|u|^{p-2}u= |u|^{q-2}u, \enspace 
\text{in} \enspace \mathbb{R}^3.
$$
The proof is completed.
\end{proof}

\section{Proof of Theorem 1.6 and Theorem 1.8}
\begin{proof}[Proof of Theorem \ref{thm1.6}]
(1) Since $q < p+\frac{p^2}{3}$, we can easily check that $f_q(t)$ $(0,+\infty)$ attains its minimum at a unique point, denoted by $t_q$. Therefore, similar to (\ref{eq3.1}), by Lemma \ref{lem2.1}, we obtain 
\begin{equation}\label{eq6.1}
i(c) = \underset{u \in S(c)}{I(u)} \geq \underset{t > 0}{\text{inf}}f_q(t) = f_q(t_q),
\end{equation}
by setting $t = |\nabla u|_p^p$.
On the other hand, set
$$
u_{\mu}(x) = \frac{c\mu^{\frac{3}{p}}}{|Q|_p}Q(\mu x),
$$
where $\mu > 0$ will be determined later. Then, $u_{\mu} \in S(c)$ and
\begin{equation}\label{eq6.2}
I(u_{\mu}) = \frac{a}{p}(c\mu)^p + \frac{b}{2p}(c\mu)^{2p}-\frac{c^{q-\frac{3(q-p)}{p}}}{p|Q|_p^{q-p}}(c\mu)^{\frac{3(q-p)}{p}} = f_q((c\mu)^p).
\end{equation}
Choosing $\mu = \frac{t_q^{\frac{1}{p}}}{c}$, i.e., $(c\mu)^p=t_q$, it follows from (\ref{eq6.2}) that $i(c) \leq I(u_{\mu}) = f_q(t_q)$. Together with (\ref{eq6.1}), we deduce that
\begin{equation}\label{eq6.3}
i(c) = f_q(t_q) = \underset{t > 0}{\text{inf}}f_q(t),
\end{equation}
and $u_c = u_{\mu}$ with $\mu = \frac{t_q^{\frac{1}{p}}}{c}$ is a minimizer of $i(c)$. 
\par
Next, we show that, up to translations, $u_c$ is the unique minimizer of $i(c)$. Indeed, if $u_0 \in S(c)$ is a minimizer, it then follows from (\ref{eq3.1}) that $I(u_0) \geq f_q(t_0)$, with $t_0 = |\nabla u_0|_p^p$, where the equality holds if and only if $u_0$ is an optimizer of (\ref{eq2.1}). This and (\ref{eq6.3}) imply that $t_0 = t_p$ and $f_q(t_0) = I(u_0)$. Thus, by Lemma \ref{lem2.1}, up to translations, $u_0$ must be the form of $u_0 = \alpha Q(\beta x)$. Using $|u_0|_p^p = c^p$, $|\nabla u_0|_p^p = t_q$ and (\ref{eq3.3}), we obtain that $
\alpha =\frac{c}{|Q|_P}\left(\frac{t_q^{\frac{1}{p}}}{c}\right)^{\frac{3}{p}} $ and $\beta = \frac{t_q^{\frac{1}{p}}}{c}$. Hence, $u_0 = u_c$.
\par
(2) When $q = p+ \frac{p^2}{3}$,
\begin{equation}\label{eq6.4}
f_q(t) = \frac{1}{p}\left(a-\frac{c^{\frac{p^2}{3}}}{p|Q|_p^{\frac{p^2}{3}}}\right)t+\frac{b}{2p}t^2.
\end{equation}
If $c \leq a^{\frac{3}{p^2}}\left| Q \right|_p$, by Theorem \ref{thm1.1}, $i(c)$ has no minimizer. If $c> a^{\frac{3}{p^2}}\left| Q \right|_p$, from (\ref{eq6.4}), we know that $f_q(t)$ $(t \in (0,+\infty))$ attains its minimum at the unique point $t_q = \frac{c^{\frac{p^2}{3}}-ap|Q|_p^{\frac{p^2}{3}}}{bp|Q|_p^{\frac{p^2}{3}}}$. Similar to the arguments of parts (1), we can prove that up to translations $u_c = \frac{c\mu^{\frac{3}{p}}}{|Q|_p}Q(\mu x)$ with $\mu = \frac{t_q^{\frac{1}{p}}}{c}$ is the unique minimizer of $i(c)$. Moreover, we have $i(c) = -\frac{b}{2p}\left(\frac{c^{\frac{p^2}{3}}-ap|Q|_p^{\frac{p^2}{3}}}{bp|Q|_p^{\frac{p^2}{3}}}\right)^2$.
\end{proof}

\begin{proof}[Proof of Theorem \ref{thm1.8}]
Firstly, similar to Lemma \ref{lem5.3}, from Definition \ref{defn1.7}, we can prove that there exists $K(c) > 0$ which can be chosen small enough such that, $I(\cdot)$ admits mountain pass geometry on $S(c)$ if (\ref{eq1.12}) holds. We assume that $K(c) < \overline{t}_q$.
\par
For any $h(s) \in \Gamma(c)$, by Lemma \ref{lem2.1}, we have
\begin{equation}\label{eq6.5}
I(h(s)) \geq f_q(|\nabla h(s)|_p^p),
\end{equation}
where equality holds if and only if $h(s) \in S(c)$ is an optimizer of (\ref{eq2.1}), i.e., up to translations,
\begin{equation}\label{eq6.6}
(h(s))(x) = \frac{c\mu^{\frac{3}{p}}}{|Q|_p}Q(\mu x) \enspace \text{for some} \enspace \mu > 0.
\end{equation}
Since $h(0) \in A_{K(c)}$ with $K(c) < \overline{t}_q$, and note that $f_q(t) > 0$ for any $t \in (0,\overline{t}_q]$, we have
\begin{equation}\label{eq6.7}
|\nabla h(0)|_p^p < \overline{t}_q < |\nabla h(1)|_p^p.
\end{equation}
By the continuity of $|\nabla h(s)|_p^p$ respect to $s$, we deduce from (\ref{eq6.5}) and (\ref{eq6.7}) that
\begin{equation}\label{eq6.8}
\underset{s\in [0,1]}{\text{max}} I(h(s)) \geq f_q(\overline{t}_q) = \underset{t>0}{\text{max}}f_q(t).
\end{equation}
Thus,
\begin{equation}\label{eq6.9}
\gamma(c) = \underset{h \in \Gamma(c)}{\text{inf}}\underset{s\in [0,1]}{\text{max}} I(h(s)) \geq f_q(\overline{t}_q).
\end{equation}
On the contrary, let $u_{\mu}(x) = \frac{c\mu^{\frac{3}{p}}}{|Q|_p}Q(\mu x)$ with $\mu = \overline{\mu}_q = \frac{\overline{t}_q^{\frac{1}{p}}}{c}$. Set $\overline{h}(s) := s^{\frac{3}{p^2}}u_{\mu}(s^{\frac{1}{p}}x)$, then we can check that $|\nabla \overline{h}(s)|_p^p = \overline{t}_qs$ and $I(\overline{h}(s))= f_q(\overline{t}_qs)$. Choosing $0 < \tilde{t}_q < \overline{t}_q$ small enough such that $\overline{h}\left(\frac{\tilde{t}_q}{\overline{t}_q}\right) \in A_{K(c)}$, and $\hat{t}_q > \overline{t}_q$ such that $I\left(\overline{h}\left(\frac{\hat{t}_q}{\overline{t}_q}\right)\right) = f_q(\hat{t}_q) < 0$. Let $h(s)=\overline{h}\left((1-s)\frac{\tilde{t}_q}{\overline{t}_q}+s\frac{\hat{t}_q}{\overline{t}_q}\right)$. Then, $h(0) = \overline{h}\left(\frac{\tilde{t}_q}{\overline{t}_q}\right) \in A_{K(c)}$ and $I(h(1)) = I\left(\overline{h}\left(\frac{\hat{t}_q}{\overline{t}_q}\right)\right) = f_q(\hat{t}_q) < 0$. This indicates that $h \in \Gamma(c)$, and
$$
\gamma(c) \leq \underset{s\in [0,1]}{\text{max}} I(h(s)) = I(u_{\overline{\mu}_q}) = f_q(\overline{t}_q).
$$
Combining with (\ref{eq6.9}), we deduce that $\gamma(c) = f_q(\overline{t}_q)$ and $\overline{u}_c = u_{\mu}(x) = \frac{c\left(\frac{\overline{t}_q^{\frac{1}{p}}}{c}\right)^{\frac{3}{p}}}{|Q|_p}Q\left(\frac{\overline{t}_q^{\frac{1}{p}}}{c} x\right) \in S(c)$ is a solution of problem (\ref{eq1.11}).
\par
Next, we prove that $\overline{u}_c$ satisfies (\ref{eq1.1})-(\ref{eq1.2}) for some $\lambda \in \mathbb{R}^-$. In view of $f_q^{\prime}(\overline{t}_q) = 0$ and $|\nabla \overline{u}_c|_p^p = \overline{t}_q$, we have
\begin{equation}\label{eq6.10}
\frac{3(q-p)c^{q-\frac{3(q-p)}{p}}}{p^2|Q|_p^{q-p}}(\overline{t}_q)^{\frac{3(q-p)}{p^2}-1} = a+b\overline{t}_q =a+b|\nabla \overline{u}_c|_p^p.
\end{equation}
Moreover, since $Q(x)$ is a solution of (\ref{eq2.2}), by $\overline{u}_c = \frac{c\left(\frac{\overline{t}_q^{\frac{1}{p}}}{c}\right)^{\frac{3}{p}}}{|Q|_p}Q\left(\frac{\overline{t}_q^{\frac{1}{p}}}{c} x\right)$, it follows that $\overline{u}_c$ satisfies
\begin{equation}
\begin{split}
& -\frac{3(q-p)c^{q-\frac{3(q-p)}{p}}}{p^2|Q|_p^{q-p}}(\overline{t}_q)^{\frac{3(q-p)}{p^2}-1}\varDelta_p \overline{u}_c - |\overline{u}_c|^{q-1}\overline{u}_c \\
& = -\left(1+\frac{(p-3)(q-p)}{p^2}\right)\frac{c^{(q-p)(1-\frac{3}{p})}}{|Q|_p^{q-p}}(\overline{t}_q)^{\frac{3(q-p)}{p^2}}|\overline{u}_c|^{p-1}\overline{u}_c.\\
\end{split}
\nonumber
\end{equation}
This together with (\ref{eq6.10}) indicates that $\overline{u}_c$ is a solution of (\ref{eq1.1})-(\ref{eq1.2}) with 
$$
\lambda = -\left(1+\frac{(p-3)(q-p)}{p^2}\right)\frac{c^{(q-p)(1-\frac{3}{p})}}{|Q|_p^{q-p}}(\overline{t}_q)^{\frac{3(q-p)}{p^2}}.
$$
\end{proof}

\section*{Acknowledgement}
The second author is supported by The 16th Postgraduated Research Innovation Project (KC-24248285).

\end{document}